\documentclass{article}
\usepackage{oldgerm,amsmath,pifont,amsfonts,amssymb,latexsym,umlaut}
\usepackage{float,xspace,calc,theorem,ifthen}
\usepackage{enumerate,psboxit,subfigure}
\usepackage[dvips]{graphicx}
\usepackage{hyperref}
\usepackage{cite}
\usepackage{changebar}
\newtheorem {lemma}{Lemma}
\newtheorem{prop}{Proposition}
\newtheorem {coro}{Corollary}

\newtheorem{hypo}{Hypothesis}
{\theorembodyfont{\rmfamily}\theoremstyle{plain}
\newtheorem {remark}{Remark}

}
{\theorembodyfont{\rmfamily}\theoremstyle{break}
}
\newcommand{\qed}{\phantom{xxxxxx}\hfill q.e.d.}
\newenvironment{proof}{{\noindent\em Proof:} }{\qed\\}

\renewcommand{\dim}{\operatorname{dim}}
\newcommand{\var}{\operatorname{var}}
\newcommand{\graph}{\operatorname{graph}}

\newcommand{\Var}{\operatorname{Var}}

\newcommand{\ph}{\varphi}

\newcommand{\tn}[1]{|\hspace*{-0.4mm}\|#1|\hspace*{-0.4mm}\|}

\newcommand{\BV}{BV}
\newcommand{\hpsi}{{\hat\psi}}
\newcommand{\1}[1]{1_{#1}}

\newcommand{\F}{{\mathcal F}}

\newcommand{\N}{{\mathcal N}}

\newcommand{\T}{{\mathcal T}}
\renewcommand{\L}{{\mathcal L}}
\renewcommand{\O}{{\mathcal O}}
\renewcommand{\o}{\text{\sl o}}
\newcommand{\rz}{{\mathbb R}}
\newcommand{\nz}{{\mathbb N}}

\newcommand{\zz}{{\mathbb Z}}

\newcommand{\eg}{{e.g.}\xspace}
\newcommand{\ie}{{i.e.}\xspace}

\newcommand{\etc}{{etc.}\xspace}

\renewcommand{\ae}{a.e.\xspace}

\begin{document}

\newcommand{\myfont}{\sf }
\title{Continuity properties of transport coefficients\\ in simple maps}

\author{Gerhard Keller$^{1}$, Phil J. Howard$^{2}$, Rainer Klages$^{2}$\\
$^{1}$Mathematisches Institut, Universit\"at Erlangen-N\"urnberg\\
Bismarckstr. $1\frac12$, 91054 Erlangen, Germany\\
$^2$School of Mathematical Sciences, Queen Mary, University of London,\\
Mile End Road, London E1 4NS, UK}

\maketitle

\begin{abstract}

We consider families of dynamics that can be described in terms of
Perron-Frobenius operators with exponential mixing properties. For
piecewise $C^2$ expanding interval maps we rigorously prove continuity
properties of the drift $J(\lambda)$ and of the diffusion coefficient
$D(\lambda)$ under parameter variation. Our main result is that
$D(\lambda)$ has a modulus of continuity of order
$\O(|\delta\lambda|\cdot|\log|\delta\lambda|)^{2})$, i.e. $D(\lambda)$
is Lipschitz continuous up to quadratic logarithmic corrections. For a
special class of piecewise linear maps we provide more precise
estimates at specific parameter values. Our analytical findings are
verified numerically for the latter class of maps by using exact
formulas for the transport coefficients. We numerically observe strong
local variations of all continuity properties.

\end{abstract}

\section{Introduction}
\label{sec:intro}

In simple deterministic dynamical systems physical quantities like
transport coefficients can be fractal functions of control parameters.
This finding was first reported for a one-dimensional piecewise linear
map lifted periodically onto the whole real line, for which the
diffusion coefficient was computed by using Markov partitions and
topological transition matrices \cite{KlDo-95,Klages-96,KlDo-99}.  A
generalization of this result was obtained for a map with both drift
and diffusion by deriving exact analytical solutions for the transport
coefficients \cite{GrKl-02,Crist-06}. Further maps modeling chemical
reaction-diffusion \cite{GaKl-98} and anomalous diffusion
\cite{KoKl-07} yielded also fractal transport coefficients. Recent
work aimed at physically more realistic models like (Hamiltonian)
particle billiards, for which computer simulations yielded transport
coefficients that are non-monotonic under parameter variation
\cite{KlBa-04}.  Ref.~\cite{Klages-07} contains a summary of this line
of research.

These results asked for a more detailed characterization of the
``fractality'' of transport coefficients. A first attempt in this
direction was reported by Klages and Klau{\ss} \cite{KlKl-03}, who
used standard techniques from the theory of fractal dimensions for
characterizing the drift and diffusion coefficients of the map studied
in \cite{GrKl-02}.  They numerically computed a non-integer box
counting dimension for these curves which varied with the parameter
interval, leading to the notion of a ``fractal fractal dimension''.
These results were questioned by Koza \cite{Koza-04}, who computed the
oscillation of these graphs at specific Markov partition parameter
values. His work suggested a dimensionality of one by conjecturing
that there exist non-trivial logarithmic corrections to the usual
power law behaviour in the oscillation.

This research reveals the need to study the parameter dependence of
transport coefficients in a rigorous mathematical setting, which can
be formulated as follows: Given a parametrized family of chaotic
dynamical systems $T_\lambda:I\to I$ on an interval $I$ with unique
invariant physical measures $\mu_\lambda$ together with a family of
sufficiently regular observables $\psi_\lambda:I\to\rz$ one has, under
suitable mixing assumptions on the systems $(T_\lambda,\mu_\lambda)$,
a law of large numbers and a central limit theorem for the partial sum
processes
$S_{\lambda,n}(x)=\sum_{k=0}^{n-1}\psi_\lambda(T_\lambda^kx)$, namely
\begin{displaymath}
  \lim_{n\to\infty}n^{-1}S_{\lambda,n}
  =J(\lambda):=\int_I\psi_\lambda(x)\,d\mu_\lambda(x)\text{\quad for $\mu_t$-\ae $x$}
\end{displaymath}
and
\begin{displaymath}
  \L(n^{-\frac12}S_{\lambda,n})\Rightarrow\N(0,2D(\lambda))
\end{displaymath}
where
$D(\lambda):=\lim_{n\to\infty}\frac1{2n}\int_I\left(\sum_{k=0}^{n-1}(\psi_\lambda(T_\lambda^kx)-J(\lambda))\right)^2d\mu_\lambda(x)$.
For suitable choices of the observables $\psi_\lambda$, the process
$S_{\lambda,n}$ is just the deterministic random walk generated by a lift of
the map $T_\lambda$ to the real axis, and $J(\lambda)$ and $D(\lambda)$ are
the drift and diffusion coefficient of this random walk respectively.

There are a few rigorous results in the literature describing the
dependence of $\mu_\lambda$ and of quantities like $J(\lambda)$ for
various classes of systems. Without going into the details they can be
summarized as follows: If the maps $T_\lambda$ and the observables
$\psi_\lambda$ depend smoothly on $\lambda$ and if the topological
conjugacy class of $T_\lambda$ is not changed when $\lambda$ is
varied, then $\mu_\lambda$ (and hence $J(\lambda)$) depends
differentiably on $\lambda$
\cite{BaSm-07,BuLi-07,Dolgopyat-04,JiLl-06,JiRu-05,Ruelle-97,Ruelle-05}.
If the topological class changes, quantities like $J(\lambda)$ may
behave less regular and have a modulus of continuity not better than
$|\delta\lambda\cdot\log|\delta\lambda||$, even for very simple maps
$T_\lambda$ like symmetric tent maps \cite{baladi-07}. On the other
hand, this modulus of continuity is the rule for systems whose
Perron-Frobenius operator (acting on a suitable space of ``regular''
densities) has a spectral gap \cite{keller-82,KL-99}.

The goal of this paper is to explicitly relate these mathematical
results to transport coefficients.  We do so by rigorously proving
continuity properties of $J(\lambda)$ and $D(\lambda)$ under parameter
variation for certain classes of deterministic maps.  In
Section~\ref{sec:general} we give a general estimate for families of
dynamics (deterministic or not), which can be described in terms of
Perron-Frobenius operators with exponential mixing properties. The
applicability of these general results to piecewise $C^2$ expanding
interval maps and in particular to the class of piecewise linear maps
discussed in \cite{KlDo-95,Klages-96,KlDo-99,GrKl-02,Klages-07} is
checked in Section~\ref{sec:checking}. The main result is that
$D(\lambda)$ has a modulus of continuity of order
$\O(|\delta\lambda|\cdot|\log|\delta\lambda|)^{2})$, i.e. $D(\lambda)$
is Lipschitz continuous up to quadratic logarithmic corrections.  In
Section~\ref{sec:transport} we summarize the general results for
transport coefficients in the special case of piecewise linear maps
and provide more precise estimates for special parameters.  Our
analytical findings are verified by numerical computations in
Section~\ref{sec:numerical}, for which we use exact analytical
formulas of the transport coefficients \cite{GrKl-02}.  Particularly,
we numerically analyze local variations of these properties. Our work
corrects and amends previous results reported in
\cite{KlKl-03,Koza-04}.

\section{The general setting}
\label{sec:general}

Let $I$ be a compact interval, $m$ normalized Lebesgue measure on $I$, $L^1_m$
the space of Lebesgue-integrable functions from $I$ to $\rz$, and $\BV\subset
L^1_m$ the space of $L^1_m$-equivalence classes of functions of bounded
variation.  We use the following simplified notation for the two corresponding
norms:
\begin{equation}
  \label{eq:norms}
  |f|_1:=\int|f|\,dm\;,\quad\|f\|:=\Var(f)
\end{equation}
where
\begin{equation}
  \label{eq:var-def}
  \Var(f)
  :=
  \sup\left\{\int f\ph'\,dm:\ph\in C^1(\rz,\rz), |\ph|_\infty\leq1\right\}
\end{equation}
is the variation of $f$ as a function from $\rz\to\rz$ (\ie extended by
$f\equiv0$ on $\rz\setminus I$). If $f$ is differentiable as a function from
$\rz\to\rz$ integration by parts shows easily that $\Var(f)=\int|f'|\,dm$.
$\Var$ is obviously a semi-norm, and as
$|f|_1\leq|f|_\infty\leq\frac12\Var(f)$, it is actually a norm. This and more
details on functions of bounded variation can be found in \cite[section
2.3]{KL-05}. The monograph \cite{baladi-book} is a comprehensive reference for
most of the background material needed in this section.

We consider a family $\T$ of nonsingular maps $T:I\to I$. \emph{Nonsingular}
means that the \emph{Perron-Frobenius operator} $P_T:L^1_m\to L^1_m$ is well
defined, \ie
\begin{equation}
  \label{eq:PF-def}
  \int P_Tf\cdot g\,dm=\int f\cdot g\circ T\,dm\quad(f\in L^1_m, g\in L^\infty_m)\;.
\end{equation}
By definition, $|P_T|_1=1$ for all $T\in\T$, and we assume
\begin{hypo}\label{hypo:norm-bounded}
  \hspace*{1cm}$C_1:=\sup\{\|P_T^n\|:T\in\T, n\in\nz\}<\infty$.
\end{hypo}
Our main assumption is that the maps in $\T$ are \emph{uniformly exponentially
  mixing} in the following sense:
\begin{hypo}\label{hypo:mixing} Each $T\in\T$ has a unique invariant probability density
  $h_T\in\BV$ (so $P_Th_T=h_T$), and there are constants
  $\gamma\in(0,1)$ and $C_2>0$ such that, for all $T\in\T$,
  \begin{equation}
    \label{eq:mixing}
    |P_T^nf|_1\leq C_2\gamma^n\Var(f)\quad\text{for all }f\in\BV\text{ with }
    \int f\,dm=0\text{ and for all } n\in\nz\;.
  \end{equation}
\end{hypo}
Observe the following consequences of Hypothesis~\ref{hypo:norm-bounded}
and~\ref{hypo:mixing}:
\begin{equation}
  \label{eq:convergence-h}
  |P_T^nf-h_T|_1\leq C_2\gamma^n(\Var(f)+2C_1)\quad\text{for all probability densities} f\in\BV
\end{equation}
and
\begin{equation}
  \label{eq:norm-bound-h}
  \Var(h_T)\leq2C_1\quad(T\in\T)\;.
\end{equation}
Indeed, $|P_T^nf-h_T|_1=|P_T^n(f-h_T)|_1\leq
C_2\gamma^n(\Var(f)+\Var(h_T))\to0$ as $n\to\infty$ by
Hypothesis~\ref{hypo:mixing} for each probability density $f\in\BV$,
and $\Var(P_T^n1)\leq C_1\Var(1)=2C_1$ by
Hypothesis~\ref{hypo:norm-bounded}.  Hence \eqref{eq:norm-bound-h}
follows from the definition \eqref{eq:var-def} of $\Var(.)$, and then
\eqref{eq:convergence-h} is an immediate consequence.

Since it is our goal to investigate the dependence of various dynamical
quantities as functions of $T\in\T$, we need to introduce a distance on $\T$.
At this stage the following one, which was already considered in
\cite{keller-82}, is most apropriate. It measures the distance between two maps
$T_1$ and $T_2$ from $\T$ in terms of a suitable norm of
$P_{T_1}-P_{T_2}$:
\begin{equation}
  \label{eq:three-bar}
  \tn{P_{T_1}-P_{T_2}}
  :=
  \sup\left\{|P_{T_1}f-P_{T_2}f|_1: f\in\BV, \|f\|\leq 1\right\}\;.
\end{equation}
This distance can be controlled in terms of a more ``hands-on'' distance
between the graphs of the maps:
\begin{equation}
  \label{eq:hands-on}
  \begin{split}
    d(T_1,T_2) := \inf\{& \epsilon>0: \exists I_\epsilon\subseteq I\text{ and
      }\exists\text{ a diffeomorphism }\sigma:I\to I\text{ s.th. }\\
      &m(I\setminus I_\epsilon)<\epsilon,\;
      T_{1}|_{I_\epsilon}=T_2\circ\sigma|_{I_\epsilon}, \text{ and }\\
      &\forall x\in I_\epsilon:\,
      |\sigma(x)-x|<\epsilon,|1/\sigma'(x)-1|<\epsilon \}\;.
  \end{split}
\end{equation}
Namely (see \cite[Lemma 13]{keller-82}),
\begin{equation}
  \label{eq:tn-hands-on}
  \tn{P_{T_1}-P_{T_2}}\leq 12\cdot d(T_1,T_2)\;.
\end{equation}
Now, as a warm-up exercise, we can prove the following estimate: for $k\geq0$
let
\begin{equation}
  \label{eq:mdoulus-of-cont}
  \ell_k:(0,\infty)\to(0,\infty),\quad\ell_k(u):=u\cdot(1+|\log u|)^k\;.
\end{equation}
\begin{lemma}\label{lemmma:density-dependence}
  There exist constants $K_1',K_1>0$ such that
  \begin{equation}
    \label{eq:density-dependence}
    |h_{T_1}-h_{T_2}|_1
    \leq
    K_1'\cdot\ell_1(\tn{P_{T_1}-P_{T_2}})
    \leq
    K_1\cdot\ell_1(d(T_1,T_2))\quad(T_1,T_2\in\T)
  \end{equation}
\end{lemma}
\begin{proof}
  Let $\tilde\eta:=\tn{P_{T_1}-P_{T_2}}$, assume without loss of generality that
  $\tilde\eta<1$, and fix $N\in\nz$. For $f\in\BV$,
  \begin{equation}\label{eq:telescoping}
    \begin{split}
      |P_{T_1}^Nf-P_{T_2}^Nf|_1
      &\leq
      \sum_{k=0}^{N-1}|P_{T_1}^{N-k-1}(P_{T_1}-P_{T_2})P_{T_2}^kf|_1
      \leq
      \sum_{k=0}^{N-1}|(P_{T_1}-P_{T_2})P_{T_2}^kf|_1\\
      &\leq
      \sum_{k=0}^{N-1}\tn{P_{T_1}-P_{T_2}}\|P_{T_2}^kf\|
      \leq
      \sum_{k=0}^{N-1}\tilde\eta\,C_1\|f\|
      \leq
      C_1N\tilde\eta\|f\|
    \end{split}
  \end{equation}
  where we used Hypothesis~\ref{hypo:norm-bounded}. Hence,
  \begin{displaymath}
    \begin{split}
      |h_{T_1}-h_{T_2}|_1
      &\leq
      |P_{T_1}^N1-P_{T_2}^N1|_1+|P_{T_1}^N(1-h_{T_1})|_1
      +|P_{T_2}^N(1-h_{T_2})|_1\\
      &\leq
      2C_1N\tilde\eta+2\cdot C_2\gamma^N(2+2C_1)
    \end{split}
  \end{displaymath}
  where we used \eqref{eq:mixing} and \eqref{eq:norm-bound-h}.  With
  $N=\lceil\frac{\log\tilde\eta}{\log\gamma}\rceil$, this is
  \eqref{eq:density-dependence}.
\end{proof}
\begin{remark}
  Even if $\T$ is a family of piecewise linear maps and if $T_1$ has
  the Markov property, this estimate can generally not be improved.
  Examples for this fact within the family of symmetric mixing tent
  maps are provided in \cite{baladi-07,mazzolena-07}.
\end{remark}

Suppose now that to each $T\in\T$ there is associated an ``observable''
$\psi_T:I\to\rz$. We make the following assumptions:
\begin{hypo}\label{hypo:psi-bounded}
  \hspace*{1cm}$C_3:=\sup\{\Var(\psi{_T}):T\in\T\}<\infty$
\end{hypo}
\begin{hypo}\label{hypo:psi-dependence}
  There is $C_4>0$ such that $|\psi_{T_1}-\psi_{T_2}|_1\leq C_4
  d(T_1,T_2)$ for all $T_1,T_2\in\T$.
\end{hypo}
Denote
\begin{equation}
  \label{eq:J-def}
  J(T):=\int_I\psi_T h_T\,dm\ .
\end{equation}
Then we have immediately from \eqref{eq:norm-bound-h} and
Lemma~\ref{lemmma:density-dependence}
\begin{coro}\label{koro:J}
  There is some $K_2>0$ such that, for all $T_1,T_2\in\T$,
  \begin{equation}
    |J(T_1)-J(T_2)|
    \leq
    K_2\cdot\ell_1(d(T_1,T_2))
  \end{equation}
\end{coro}

$J(T)$ is the ``drift'' of the partial sum process
\begin{displaymath}
  S_{T,n}:=\sum_{k=0}^{n-1}\psi_T\circ T^k
  =
  n\,J(T)+\sum_{k=0}^{n-1}\hpsi_T\circ T^k
\end{displaymath}
under the invariant measure $h_Tm$, where $\hpsi_T=\psi_T-J(T)$. Observe that
\begin{equation}
  \label{eq:hpsi-bound}
  \Var(\hpsi_T)\leq 2C_3,\quad
  |\psi_{T_1}-\psi_{T_2}|\leq2C_4\,d(T_1,T_2)\quad\text{for all }T,T_1,T_2\in\T.
\end{equation}
In view of Hypothesis~\ref{hypo:mixing} we can also define the
``diffusion coefficient''\footnote{This is the convention in the
  physics literature. In the mathematics literature one would rather
  call $2D(T)$ the diffusion coefficient.}  of this process:
\begin{equation}
  \label{eq:variance-def}
  \begin{split}
    D(T)
    &:=
    \lim_{n\to\infty}\frac1{2n}\int\left(\sum_{k=0}^{n-1}\hpsi_T\circ
    T^k\right)^2h_T\,dm\\
    &=
    \frac12\int\hpsi_T^2\,h_T\,dm
    +\sum_{n=1}^\infty\int\hpsi_T\cdot\hpsi_T\circ T^n\,h_T\,dm\\
    &=
    \frac12\int\hpsi_T^2\,h_T\,dm
    +\sum_{n=1}^\infty\int P_T^n(\hpsi_Th_T)\,\hpsi_T\,dm
  \end{split}
\end{equation}
Even more, we have the central limit theorem
\begin{equation}
  \label{eq:cml}
  \L(n^{-\frac12}(S_{T,n}-nJ(T))\Rightarrow\N(0,2D(T))\text{ as }n\to\infty\;,
\end{equation}
see \eg \cite{keller-80,HK-82,RE-83}. Among physicists
(\ref{eq:variance-def}) is known as the Taylor-Green-Kubo formula
for diffusion \cite{Klages-07}. For the dependence of $D(T)$ on $T$ we
prove:
\begin{prop}\label{lemma:D}
 There is some $K_3>0$ such that, for all $T_1,T_2\in\T$,
 \begin{equation}
   \label{eq:D-dependence}
   |D(T_1)-D(T_2)|\leq K_3\cdot\ell_2(d(T_1,T_2))
 \end{equation}
\end{prop}
\begin{proof}
  Observe first that, for all $\psi,h\in\BV$,
  \begin{equation}
    \label{eq:D-proof-0}
    \Var(\psi h)
    \leq
    \Var(\psi)|h|_\infty+\Var(h)|\psi|_\infty
    \leq
    \Var(\psi)\Var(h)
    \;.
  \end{equation}
  Indeed, for differentiable $u$ we have $\Var(u)=\int|u'|\,dm$, so for
  differentiable $h$ and $\psi$ \eqref{eq:D-proof-0} follows from
  the product rule of diferentiation. General $h$ and $\psi$ are
  then approximated using mollifiers.
  It follows that, in view of \eqref{eq:mixing},
  \begin{equation}
    \label{eq:D-proof-1}
    \begin{split}
      \left|\int P_T^n(\hpsi_Th_T)\,\hpsi_T\,dm\right|
      &\leq
      |P_T^n(\hpsi_Th_T)|_1\cdot|\hpsi_T|_\infty
      \leq
      C_2C_3\gamma^n\Var(\hpsi_T h_T)\\
      &\leq
      4C_1C_2C_3^2\gamma^n\;.
    \end{split}
  \end{equation}
  Let $\tilde\eta=\|P_{T_1}-P_{T_2}\|$ as before, denote $\eta:=d(T_1,T_2)$
  (so that $\tilde\eta\leq12\eta$), and fix
  $N=\lceil\frac{\log\eta}{\log\gamma}\rceil$.  For all $T_1,T_2\in\T$, eq.
  \eqref{eq:D-proof-1} implies
  \begin{equation}
    \label{eq:D-proof-2}
    2\sum_{n=N}^\infty\left|\int P_{T_1}^n(\hpsi_{T_1}h_{T_1})\,\hpsi_{T_1}\,dm-
      \int P_{T_2}^n(\hpsi_{T_2}h_{T_2})\,\hpsi_{T_2}\,dm\right|
    \leq
    \frac{16C_1C_2C_3^2}{1-\gamma}\eta\;.
  \end{equation}

  For $0\leq n<N$ we use a different estimate. We decompose
  \begin{equation}
    \label{eq:D-proof-3}
    \int P_{T_1}^n(\hpsi_{T_1}h_{T_1})\,\hpsi_{T_1}\,dm-
    \int P_{T_2}^n(\hpsi_{T_2}h_{T_2})\,\hpsi_{T_2}\,dm
    =\Delta_1^n+\Delta_2^n+\Delta_3^n+\Delta_4^n
  \end{equation}
  where
  \begin{equation}
    \label{eq:D-proof-4}
    \begin{split}
      |\Delta_1^n|
      :=&
      \left|\int P_{T_1}^n(\hpsi_{T_1}h_{T_1})\,(\hpsi_{T_1}-\hpsi_{T_2})\,dm\right|
      \leq
      \frac12\Var(P_{T_1}^n(\hpsi_{T_1}h_{T_1}))\,|\hpsi_{T_1}-\hpsi_{T_2}|_1\\
      \leq&
      4C_1^2C_3C_4\eta
    \end{split}
  \end{equation}
and
  \begin{equation}
    \label{eq:D-proof-5}
    \begin{split}
      |\Delta_2^n|
      :=&
      \left|\int(P_{T_1}^n-P_{T_2}^n)(\hpsi_{T_1}h_{T_1})\,\hpsi_{T_2}\,dm\right|
      \leq
      \tn{P_{T_1}^n-P_{T_2}^n}\Var(\hpsi_{T_1}h_{T_1})|\hpsi_{T_2}|_\infty\\
      \leq&
      4C_1C_3^2\tn{P_{T_1}^n-P_{T_2}^n}
      \leq
      4(C_1C_3)^2n\tilde\eta
    \end{split}
  \end{equation}
  where the last inequality follows from eq.~\eqref{eq:telescoping}.
 Next,
\begin{equation}
  \label{eq:D-proof-7a}
  \begin{split}
    |\Delta_3^n| :=&
    \left|\int P_{T_2}^n\big((\hpsi_{T_1}-\hpsi_{T_2})h_{T_1}\big)\hpsi_{T_2}\,dm\right|
    \leq
    |\hpsi_{T_1}-\hpsi_{T_2}|_1\,|h_{T_1}|_\infty\,|\hpsi_{T_2}|_\infty\\
    \leq&
    4C_1C_3C_4\eta
  \end{split}
\end{equation}
and
\begin{equation}
  \label{eq:D-proof-7b}
  \begin{split}
    |\Delta_4^n| :=&
    \left|\int P_{T_2}^n\big(\hpsi_{T_2}(h_{T_1}-h_{T_2})\big)\hpsi_{T_2}\,dm\right|
    \leq
    |h_{T_1}-h_{T_2}|_1\,|\hpsi_{T_2}|_\infty\,|\hpsi_{T_2}|_\infty\\
    \leq&
    C_3^2K_1\ell_1(\tilde\eta)
  \end{split}
\end{equation}

>From \eqref{eq:D-proof-4} - \eqref{eq:D-proof-7b} we see that
\begin{equation}
  \label{eq:D-proof-8}
  |\Delta_1^n|+|\Delta_2^n|+|\Delta_3^n|+|\Delta_4^n|
  \leq
  \tilde K(n\eta+\ell_1(\eta))
\end{equation}
for some constant $\tilde K>0$. Hence, in view of \eqref{eq:D-proof-2} and the
choice of $N$,
\begin{equation}
  \label{eq:D-proof-9}
  \begin{split}
    |D(T_1)-D(T_2)|
    &\leq
    \frac{4C_1C_2C_3^2}{1-\gamma}\eta+
    \tilde K\big(\ell_1(\eta)+2\sum_{n=1}^{N-1}(n\eta+\ell_1(\eta))\big)\\
    &\leq
    \frac{4C_1C_2C_3^2}{1-\gamma}\eta+
    \tilde K\big(N^2\eta+(2N-1)\ell_1(\eta)\big)\\
    &\leq
    K_3\cdot\ell_2(\eta)
  \end{split}
\end{equation}
for a suitable constant $K_3$.
\end{proof}

\begin{remark}
  Quite often slightly stronger forms of Hypotheses~\ref{hypo:mixing}
  and~\ref{hypo:psi-dependence} are satisfied, where the mixing
  assumption \eqref{eq:mixing} is replaced by
  \begin{equation}
    \label{eq:BV-mixing}
    \Var(P_T^nf)\leq C_2'\gamma^n\Var(f)\quad\text{for all }f\in\BV\text{ with }
    \int f\,dm=0\text{ and } n\in\nz
  \end{equation}
and the assumption on the $T$-dependence of $\psi_T$ is strengthened to
\begin{equation}
  \label{eq:strong-psi-dependence}
  \Var(\psi_{T_1}-\psi_{T_2})\leq C_4'd(T_1,T_2)
  \quad\text{for all $T_1,T_2\in\T$.}
\end{equation}
An inspection of the above estimates shows that
$|\Delta_1^n|\leq4C_1C_2'C_3C_4\gamma^n\eta$ and $|\Delta_2^n|\leq
  4C_1C_2'C_3(1-\gamma)^{-1}\eta$ if \eqref{eq:BV-mixing} is assumed. If
  additionally \eqref{eq:strong-psi-dependence} is assumed, then $|\Delta^n_3|$
  can be estimated as follows: Let
  $\alpha:=\int_I(\hpsi_{T_1}-\hpsi_{T_2})h_{T_1}dm$. Then
\begin{equation}
  \begin{split}
    |\Delta_3^n|
    &=
    \left|\int P_{T_2}^n\big((\hpsi_{T_1}-\hpsi_{T_2})h_{T_1}-\alpha h_{T_2}\big)\hpsi_{T_2}\,dm\right|
    \leq
    C_2C_3\,\gamma^n\,\Var\left((\hpsi_{T_1}-\hpsi_{T_2})h_{T_1}-\alpha
    h_{T_2}\right)\\
  &\leq
    C_1C_2C_3C_4'(4+2)\gamma^n\eta\ .
  \end{split}
\end{equation}
Hence, $\sum_{n=0}^{N-1}|\Delta^n_1|+|\Delta^n_3|=\O(\eta)$ uniformly in
$N$ and $\sum_{n=0}^{N-1}|\Delta^n_2|=\O(N\eta)=\O(\ell_1(\eta))$. But we see no
way, in general, to bound the $\Delta_4^n$-terms in a similar way. However,
for particular families of maps (which are all topologically conjugate), we will
see in subsection~\ref{subsec:integer-slope} that $\Delta^n_4=0$ for all $n$ and that the estimate for
$|\Delta^n_2|$ can be made more precise.
\end{remark}

\section{Checking Hypothesis~\ref{hypo:norm-bounded} and~\ref{hypo:mixing}}
\label{sec:checking}

\subsection{General piecewise expanding maps}

In this subsection we show how the general Hypothesis~\ref{hypo:norm-bounded}
and~\ref{hypo:mixing} can be verified in the more particular setting when $\T$
is a parametrized family of piecewise twice continuously differentiable and
expanding interval maps. So, from now on, we look at the following setting:
\begin{gather}
  \text{$\Lambda\subset\rz^d$ is a compact parameter space,
    $\T=\{T_\lambda:\lambda\in\Lambda\}$, and}\tag{T1}
  \label{T1}\\
  \text{there is some $L>0$ such that $d(T_{\lambda_1},T_{\lambda_2})\leq
    L|\lambda_1-\lambda_2|$ for all
    $\lambda_1,\lambda_2\in\Lambda$.}\tag{T2}
  \label{T2}
\end{gather}

We start with an abstract result which reduces Hypothesis~\ref{hypo:mixing}
essentially to a uniform Lasota-Yorke type inequality.
\begin{lemma}\label{lemma:LY}
  Assume \eqref{T1} and \eqref{T2}. Then Hypothesis~\ref{hypo:norm-bounded}
  and~\ref{hypo:mixing} are valid if the transformations $T\in\T$ are mixing
  and satisfy a \emph{uniform Lasota-Yorke type inequality}: there are
  constants $C_5,C_6>0$ and $\alpha\in(0,1)$ such that
\begin{equation}
  \Var(P_T^nf)
  \leq
  C_5\alpha^n\Var(f)+C_6|f|_1\quad\text{for all }T\in\T, n\in\nz, f\in\BV.
  \tag{LY}
  \label{eq:LY}
\end{equation}
\end{lemma}

\begin{proof}
  As $|P_T|_1=1$ and $|f|_1\leq\frac12\|f\|$, it is straightforward to check
  that Hypothesis~\ref{hypo:norm-bounded} holds with $C_1=C_5+\frac12C_6$.

  We turn to Hypothesis~\ref{hypo:mixing}. Note first that, because of
  \eqref{T1} and \eqref{T2}, it suffices to show that for each
  $\lambda\in\Lambda$ there are $\delta(\lambda)>0$, $C_2(\lambda)>0$ and
  $\gamma(\lambda)\in(0,1)$ such that \eqref{eq:mixing} holds with these
  constants for all $T_{\lambda_1}$ with
  $|\lambda_1-\lambda|<\delta(\lambda)$. But this is guaranteed by Corollary 2
  and Remark 1c  in \cite{KL-99}.
\end{proof}

Our next task is to give sufficient conditions for \eqref{eq:LY}. To this end we
specialize further and assume from now on that our maps are piecewise
expanding (PE) maps in the following sense:
\begin{equation}
  \begin{split}
    &\text{For each }\lambda\in\Lambda\text{ there is a finite partition }
  (I^1_\lambda,\dots,I^{N_\lambda}_\lambda)\text{ of }I\text{ into
  sub-}\\ &\text{intervals such
  that all }T_\lambda|_{I^j_\lambda}\text{ are monotone, $C^2$, and }
  \kappa_\lambda:=\inf|T_\lambda'|>2\,.
  \end{split}
  \tag{PE}
  \label{eq:PE}
\end{equation}
Already in \cite{LY-73} it was proved that each individual (PE)-map (even if
$1<\kappa_\lambda\leq2$) satisfies \eqref{eq:LY} with constants $C_5,C_6,\alpha$
depending on the map. For parametrized families of maps one can generally find
uniform constants, but there are counterexamples where this is not
possible \cite{keller-82,blank-93,BlKe-97}. Under the
above assumption $\inf_{\lambda\in\Lambda}\kappa_\lambda>2$ one can, however,
give simple sufficient conditions ensuring the uniform LY-inequality.  The
proof in \cite{LY-73} (see also \cite[Proposition 2.1]{KL-05}) shows
 \begin{equation}
   \label{eq:LY-detail}
   \Var(P_{T_\lambda}f)
   \leq
   \frac 2{\kappa_\lambda}\Var(f)+(D_\lambda+E_\lambda)|f|_1
 \end{equation}
 where
  \begin{equation}
    D_\lambda=\sup_{x}\left|\left(\frac1{T_\lambda'(x)}\right)'\right|\;,\quad
    E_\lambda=\frac2{\kappa_\lambda\min_j|I_\lambda^j|}\;.
  \end{equation}
  From this \eqref{eq:LY} follows with $\alpha=\frac2{\kappa_\lambda}$,
  $C_5=1$, and
  $C_6=\sup_\lambda\frac{\kappa_\lambda}{\kappa_\lambda-2}(D_\lambda+E_\lambda)$
  provided this supremum is finite.

 In some cases of interest the $E_\lambda$,
 $\lambda\in\Lambda$, are not bounded because there are arbitrarily short
 monotonicity intervals. In such situations, \emph{ad hoc}
 arguments are needed. We give an example in the next section.

\subsection{Piecewise linear modulo 1 maps}
\label{subsec:Tab}
We now look at a particular model dealt with in
\cite{FlLa-97a,FlLa-97b,hofbauer-80,hofbauer-81} from a mathematical
perspective and in \cite{KlDo-95,Klages-96,KlDo-99,GrKl-02,Klages-07}
from a physics point of view. Let $I=[-\frac12,\frac12]$,
$\Lambda=[a_0,a_1]\times[-\frac12,\frac12]$ for some constants $2<a_0<a_1$,
and for $\lambda=(a,b)\in\Lambda$ consider
\begin{equation}
  \label{eq:pwl}
  T_\lambda(x)=ax+b\mod (\zz-\frac12)\;.
\end{equation}
Hofbauer \cite{hofbauer-80} showed that these maps have always a
unique invariant probability density\footnote{Indeed, Hofbauer shows
this for the maximal measure of such maps, but since these maps have
constant slope, the maximal measures are just the absolutely
continuous ones. For numerical results on the probability densities
associated with these measures and how they change under parameter
variation see \cite{Klages-96}.}, but although
these maps received further attention also in the mathematical literature
\cite{hofbauer-81,FlLa-97a,FlLa-97b}, it is not so easy to draw
Hypothesis~\ref{hypo:mixing} from these sources.  Therefore we will take up a
rather direct computation made in \cite{keller-99} to prove, without having to
rely on the compactness assumption \eqref{T1}, the following lemma.
\begin{lemma}\label{lemma:var-contraction}
  Let $\lambda=(a,b)\in\Lambda$. Then
  \begin{equation}
    \label{eq:var-contraction}
    \Var(P_{T_\lambda}f)
    \leq
    \frac2a\Var(f)+2\left|\int f\,dm\right|\qquad\text{for all $f\in\BV$ and $n\in\nz$.}
  \end{equation}
(This implies immediately Hypothesis~\ref{hypo:norm-bounded} with
$C_1=1+\frac{a_0}{a_0-2}$ and Hypothesis~\ref{hypo:mixing} as well as its
strengthening \eqref{eq:BV-mixing} with
$\gamma=\frac2{a_0}$ and $C_2=C_2'=1$.)
\end{lemma}

\begin{proof}
  Denote by $\F$ the family of all $C^1$-functions $\ph:\rz\to\rz$ with
  $|\ph|_\infty\leq1$. In \cite[eq. (11)]{keller-99} a number
  $\Gamma(\ph,s)\geq0$ is defined for each pair of $\ph\in\F$ and $s\in\rz$.
  In view of \cite[eq. (13) and (14)]{keller-99} it suffices to show that for
  each $\ph\in\F$ there is some $s\in[-2,2]$ such that
  $\Gamma(\ph,s)\leq\frac 2a$.\footnote{$\var(.)$ in \cite{keller-99} is the same as
    $\Var(.)$ here. This is different from the use of $\var(.)$ in
    \cite{KL-05}.}  \footnote{Following this reference precisely, the reader
    will notice that instead of the factor $\frac2a$ one gets the factor
    $\frac2a+\frac12V_g$. But $V_g=0$ for piecewise linear transformations as
    noticed at the bottom of \cite[p.1779]{keller-99}.}  We will show that this
  is the case for $s=s_\ph:=\ph(\frac12)-\ph(-\frac12)$.

  Let $p=\lceil\frac{a+1}2-b\rceil-1$ and $q=\lceil\frac{a+1}2+b\rceil-1$.
  Then $-p-\frac12\leq T_\lambda(-\frac12)<-p+\frac12$,
  $q-\frac12<T_\lambda(\frac12)\leq q+\frac12$, and $T_\lambda$ has
  monotonicity intervals $(A_k,B_k)$, $k=-p,\dots,q$, where
  \begin{align*}
    &A_{-p}=-\frac12,\quad&A_k=a^{-1}(k-\frac12-b)\quad(k=-p+1,\dots,q)\\
    &B_q=\frac12,&B_k=a^{-1}(k+\frac12-b)\quad(k=-p,\dots,q-1)
  \end{align*}
  In order to estimate $\Gamma(\ph,s)$ in \cite[eq. (11)]{keller-99} one has
  to evaluate certain terms $U_k:=\psi(A_k)-\ell_kg_k(A_k)-sA_k$ and
  $V_k:=\psi(A_k)-\ell_kg_k(A_k)-sB_k$. In our case, $g_k(A_k)=a^{-1}$, as $g_k$ is
  the inverse of the derivative of the $k$-th monotone branch. The quantity
  $\ell_k$ is an abbreviation for $\ph(T_\lambda A_k)$,  where $T_\lambda
    A_k$ denotes a limit from the right. Therefore, using the formula on the bottom of
  \cite[p.1779]{keller-99}, we obtain
  \begin{equation}
    \label{eq:evaluation1}
    \begin{split}
    a\,U_k=&a\cdot(\psi(A_k)-\ell_kg_k(A_k)-s_\ph A_k)\\
    =&
    -\ph(-\frac a2+b+p)+(k+p)\big(\ph(\frac12)-\ph(-\frac12)\big)
    -s_\ph\cdot(k-\frac12-b)\\
    =&
    -\ph(-\frac a2+b+p)+s_\ph\cdot(p+\frac12+b)\quad\text{ if }k>-p\ ,\\
    a\,U_{-p}=&
    -\ph(-\frac a2+b+p)+s_\ph\cdot\frac a2\ ,
    \end{split}
  \end{equation}
  and similarly,
  \begin{equation}
    \label{eq:evaluation2}
    \begin{split}
      a\,V_k
      =&
      a\cdot(\psi(A_k)-\ell_kg_k(A_k)-s_\ph B_k)\\
      &=
      -\ph(-\frac a2+b+p)+s_\ph\cdot(p-\frac12+b)\quad\text{ if }k<q\ ,\\
      a\,V_q
      =&
      -\ph(-\frac a2+b+p)+s_\ph\cdot(p+q-\frac a2)\ .
    \end{split}
  \end{equation}
  It follows that $|V_k-U_{k'}|\leq a^{-1}|s_\ph|$ for all
  $k,k'\in\{-p,\dots,q\}$. Hence, by \cite[eqs. (11) and (12)]{keller-99},
  \begin{equation}
    \label{eq:evaluation3}
    \Gamma(\ph,s_\ph)
    \leq
    a^{-1}+\frac12 a^{-1}|s_\ph|
    \leq
    \frac2a\ .
  \end{equation}
\end{proof}

Next we check assumption \eqref{T2} on the Lipschitz dependence of the maps on
the parameters, so we estimate $d(T_{a,b},T_{a',b'})$. For the proof we extend
the maps to the whole real line (keeping the same names) by applying
definition \eqref{eq:pwl} to all $x\in\rz$.

Suppose $a'<a$ and denote by $A_k$ and $A_k'$, $k=-p+1,\dots,q$, the
discontinuity points of the two maps as introduced in the proof of
Lemma~\ref{lemma:var-contraction}. Consider the linear map $L:\rz\to\rz$,
$L(x)=(ax+b-b')/a'$ and observe that $L(A_k)=A_k'$ $(-p<k\leq q)$ and
$a'L(x)+b'=ax+b$ for all $x\in\rz$. Let $[u,v]:=I\cap L^{-1}(I)$ and
$I_0:=[u+\delta,v-\delta]$ for some arbitrarily small $\delta>0$. Define
$\sigma:I\to I$ by $\sigma(x)=L(x)$ if $x\in I_0$ and extend $\sigma$ to a
diffeomorphism of $I$. Then
\begin{itemize}
\item $m(I\setminus I_0)\leq(1-\frac{a'}a)+|b'-b|/a+2\delta
  \leq a_0^{-1}(|a'-a|+|b'-b|)+2\delta$
\item $|\sigma(x)-x|=|L(x)-x|\leq\frac12|\frac a{a'}-1|+|b-b'|/a'\leq a_0^{-1}(\frac12|a-a'|+|b-b'|)$ for all $x\in I_0$
\item $|1/\sigma'(x)-1|=|\frac{a'}a-1|\leq a_0^{-1}|a'-a|$ for all $x\in I_0$
\item $T_{a'}(\sigma(x))=T_a(x)$ for all $x\in I_0$.
\end{itemize}
So $d(T_a,T_{a'})\leq a_0^{-1}(|a-a'|+|b-b'|)$ as $\delta>0$ could be
chosen arbitrarily small,

\section{Transport coefficients}
\label{sec:transport}

We apply the results of the previous sections to determine transport
coefficients of the deterministic random walks generated by the maps
$T_\lambda=T_{a,b}$ from subsection~\ref{subsec:Tab}. The random
walks in question are $S_{\lambda,n}=\sum_{k=0}^{n-1}\psi_\lambda\circ
T_\lambda^k$ with
\begin{equation}
  \label{eq:psi-def}
  \psi_\lambda(x)=
  (a-1)x+b\;.
\end{equation}
It is an easy exercise to see that Hypothesis~\ref{hypo:psi-bounded}
and~\ref{hypo:psi-dependence} as well as their strengthening
\eqref{eq:strong-psi-dependence} are satisfied: $\Var(\psi_{a,b})\leq
2(a-1+|b|)<2(a_1+1)=:C_3$ and
$|\psi_{a,b}-\psi_{a',b'}|_1\leq\frac12\Var(\psi_{a,b}-\psi_{a',b'})\leq|a-a'|+|b-b'|$.

For later use we note that the maps $T_{a,b}$ and $T_{a,-b}$ are conjugate in the sense that
$T_{a,b}(-x)=-T_{a,-b}(x)$, in particular $h_{a,-b}(-x)$ is also an invariant
density for $T_{a,b}$ and, by uniqueness, $h_{a,b}(x)=h_{a,-b}(-x)$.


We first note the following explicit form of the drift:
\begin{equation}
  \label{eq:drift}
  J(\lambda)=J(a,b)
  :=
  \int\psi_{a,b} h_{a,b}\,dm
  =
  b+(a-1)\int xh_{a,b}(x)\,dx\;.
\end{equation}
As noted above,
$h_{a,0}(x)=h_{a,0}(-x)$. Hence $J(a,0)=0$.\footnote{In
\cite{Klages-96} it was conjectured, based on analyzing these deterministic
random walks in terms of Markov partitions, that for $b=0$ and $a>2$
the maps $T_\lambda$ exhibit a central limit theorem and that a diffusion
coefficient exists, which is confirmed by (\ref{eq:cml}).}

\subsection{Upper bounds for the modulus of continuity of the drift and the
  diffusion coefficient}
Now we apply
Corollary~\ref{koro:J} and Proposition~\ref{lemma:D} to our setting.

\begin{prop}\label{prop:main}
  For the family of maps $(T_\lambda: \lambda\in\Lambda)$ defined
  above, there are constants $K_3,K_4>0$ such that the drift
  $J(\lambda):=J(T_{\lambda})$ and the diffusion coefficient
  $D(\lambda):=D(T_{\lambda})$, $\lambda=(a,b)$, satisfy
  \begin{align}
    |J(\lambda)-J(\lambda')|
    &\leq
    K_3\cdot|\lambda-\lambda'|\cdot\big(1+\big|\log|\lambda-\lambda'|\big|\big)\quad(\lambda,\lambda'\in\Lambda)\\
    |D(\lambda)-D(\lambda')|
    &\leq
    K_4\cdot|\lambda-\lambda'|\cdot\big(1+\big|\log|\lambda-\lambda'|\big|\big)^2\quad(\lambda,\lambda'\in\Lambda)\;.
  \end{align}
\end{prop}
\begin{coro}\label{coro:main}
  \begin{enumerate}[a)]
  \item The graph of $D:\Lambda\to\rz$ has box- and Hausdorff-dimension $2$.
  \item For each $b\in\rz$, the graph of $D_b:[a_0,a_1]\to\rz$,
    $D_b(a)=D(a,b)$, has box- and Hausdorff-dimension $1$.
  \item For each $a>2$, the graph of $D_a:[-\frac12,\frac12]\to\rz$,
    $D_a(b)=D(a,b)$, has box- and Hausdorff-dimension $1$.
  \end{enumerate}
\end{coro}
\begin{proof}
  Denote by $\dim_B$ and $\dim_H$ the box and Hausdorff dimension,
  respectively.
  Obviously,
  $2\leq\dim_H(\graph(D))\leq\dim_B(\graph(D))$. So it
  remains to show that $\dim_B(\graph(D))\leq2$. To this end subdivide the
  rectangle $\Lambda$ into little squares of equal size $\approx N^{-1}$. For
  each such square $Q$ we have
  \begin{equation}
\label{eq:osc}
  \max\{D(\lambda):\lambda\in Q\}-\min\{D(\lambda):\lambda\in Q\}\leq
  K_4N^{-1}(1+\log N)^2\;.
  \end{equation}
  Hence,
  \begin{equation}
    \dim_B(\graph(D))
    \leq
    \limsup_{N\to\infty}\frac{\log(K_4N^2(1+\log N)^2)}{\log N}=2\;.
  \end{equation}
The two other claims are proved in the same way.
\end{proof}


\begin{remark}\label{remark:Koza}
  Corollary 2 has already been conjectured by Koza \cite{Koza-04}. His
  conjecture was based on calculating the pointwise
  Minkowski-Bouligand dimension for algebraic Markov partition
  parameter values of this family of maps by using the exact solutions
  for drift and diffusion coefficient given in \cite{GrKl-02}. This
  led him to conclude that the oscillation \cite{Tr-95} of
  $D(\lambda)$ is linear in the size of the subinterval multiplied
  with a logarithmic term, cp.\ (20) of \cite{Koza-04} with
  (\ref{eq:osc}) above. The exponent of this logarithmic
  correction was found to be either one or two depending on the type
  of Markov partition.
\end{remark}

\subsection{A closer look at maps with integer slope}
\label{subsec:integer-slope}
We finish this section with a closer look at the functions $D_a(b)$ when $a$
is an integer larger than $2$.  In this case $T=T_{a,b}$ can be seen as an
$a$-fold covering linear circle map, so it leaves Lebesgue measure invariant.
Therefore $h_T=1$ for all such $T$ and the estimates from the proof of
Proposition~\ref{lemma:D} simplify drastically: Fix $a\in\{3,4,5,\dots\}$.
Then $J(a,b)=b$ for all $b$ (see \eqref{eq:drift} and
\cite{GrKl-02,Klages-07}), $\hpsi_{a,b}(x)=(a-1)x=:\hpsi_a(x)$ is the same for
all $b$, and denoting $T_1=T_{a,b}$ and $T_2=T_{a,b'}$ we can replace
estimates \eqref{eq:D-proof-3} - \eqref{eq:D-proof-8} by
\begin{equation}\label{eq:a-1square}
  \begin{split}
    &\int_I P_{T_2}^n(\hpsi_{T_2}h_{T_2})\,\hpsi_{T_2}\,dm-
    \int_I P_{T_1}^n(\hpsi_{T_1}h_{T_1})\,\hpsi_{T_1}\,dm\\
    =& \int_I \hpsi_a(x)\,\hpsi_a(T_2^nx)\,dx
    -\int_I \hpsi_a(x)\,\hpsi_a(T_1^nx)\,dx\\
    =&
    (a-1)^2\left(\int_Ix\,T_{a,b'}^n(x)\,dx-\int_Ix\,T_{a,b}^n(x)\,dx\right)
  \end{split}
\end{equation}
To evaluate this difference assume henceforth that $0\leq b\leq\frac12$. There
is no loss in doing so, because $T_{a,b}$ and $T_{a,-b}$ are conjugate as
obseved above. Define the
``rotation'' $R:I\to I$ by $R(x)=x-\frac b{a-1}\mod(\zz-\frac12)$. It
conjugates $T_{a,b}$ to $T_{a,0}$, namely
\begin{displaymath}\label{eq:conjugacy}
  R(T_{a,0}^nx)=T_{a,b}^n(Rx)\quad\text{for all $x\in I$ and $n\in\nz$.}
\end{displaymath}
Therefore, denoting $\hat b=-\frac12+\frac b{a-1}$ and
$\chi_b(x)=\1{[-\frac12,\hat b)}(x)-\frac b{a-1}$,
\begin{displaymath}
  \int_Ix\,T_{a,b}^n(x)\,dx
  =
  \int_IR(x)\,R(T_{a,0}^nx)\,dx
  =
  \int_I(x+\chi_b(x))(T_{a,0}^nx+\chi_b(T_{a,0}^nx))\,dx
\end{displaymath}
so that
\begin{displaymath}
  \begin{split}
    &\int_Ix\,T_{a,b'}^n(x)\,dx-\int_Ix\,T_{a,b}^n(x)\,dx\\
    =& \int_IP_{a,0}^nx\cdot(\chi_{b'}(x)-\chi_{b}(x))\,dx
    +\int_IP_{a,0}^n(\chi_{b'}-\chi_{b})(x)\cdot x\,dx\\
    & +\int_IP_{a,0}^n\chi_{b'}(x)\cdot(\chi_{b'}(x)-\chi_{b}(x))\,dx +
    \int_IP_{a,0}^n(\chi_{b'}-\chi_{b})(x)\cdot\chi_{b}(x)\,dx
  \end{split}
\end{displaymath}
As $\int_I\chi_{b'}(x)\,dx=0$ and $\Var(\chi_{b})=2$ for all $b$, the third term
is of order $\O((2/a)^n|b'-b|)$ by Lemma~\ref{lemma:var-contraction}. As
$P_{a,0}^nx=a^{-n}x$, the first term is at most of the same order.  Therefore
their sums over all $n$ are of the order $\O(|b'-b|)$.

We turn to the two remaining terms. Their sum from $n=0$ to $\infty$ is of the
form
\begin{equation}\label{eq:2-3-sum}
  \sum_{n=0}^{\infty}\int_I P_{a,0}^n(\chi_{b'}-\chi_{b})(x)\cdot g_b(x)\,dx
  =
  \sum_{n=0}^{\infty}\int_I(\chi_{b'}-\chi_{b})(x)\cdot g_b(T_{a,0}^nx)\,dx
\end{equation}
with $g_b(x)=x+\chi_{b}(x)$.  Let $\delta=\frac{b'-b}{a-1}$. As
$\int_{-1/2}^{1/2}(\chi_{b'}-\chi_{b})(x)\,dx=0$,
$\Var(\chi_{b'}-\chi_{b})\leq4$, and $|g_b|\leq2$, the $n$-th integral is of
order $\O((2/a)^n)$.  Hence the sum from
$n=N_\delta:=\lceil\frac{\log|\delta|^{-1}}{\log a}\rceil$ to $\infty$ is of
order $|\delta|$, and it remains to estimate the sum from $n=0$ to
$N_\delta-1$. For these $n$ we have $a^n|\delta|\leq1$.

We start with the special case $b=0$ where we have $\chi_b=0$, so $g_b(x)=x$.
Then the $n$-th term in the sum
\eqref{eq:2-3-sum} evaluates to
\begin{displaymath}
  \int_{-\frac12}^{-\frac12+\delta}T_{a,0}^n(x)\,dx
  =
  \begin{cases}
    \delta\cdot(-\frac12+\frac12a^n\delta)&\text{ if $a$ is odd}\\
    \delta\cdot(\frac12a^n\delta)&\text{ if $a$ is even and $n\geq1$}
  \end{cases}
\end{displaymath}
It follows that the sum from $n=0$ to $N_\delta-1$ in \eqref{eq:2-3-sum} is of
the order $\O(|\delta|)$ if $a$ is even and that it is
$\delta\cdot(-\frac{\log|\delta|^{-1}}{2\log a}+\O(1))$ if $a$ is odd.

Consider next the case $b=\frac12$. As $b=-\frac12$ gives rise
to the same map, the we may assume w.l.o.g. that $-\frac12<b'<b$, \ie $\delta<0$.  If
  $a$ is odd, then $T_{a,0}\hat b=\hat b+\frac12$ and $T_{a,0}(\hat
  b+\frac12)=\hat b$. Therefore
  \begin{displaymath}
    \begin{split}
      \int_I(\chi_{b'}-\chi_{b})(x)\cdot g_b(T_{a,0}^nx)\,dx
      &=
     \int_{\hat{b}}^{\hat{b'}}g_b(T_{a,0}^nx)\,dx\\
     &=
      \begin{cases}
        \int_{0}^{\delta}(\hat b+a^nt)+\chi_b(\hat b+a^nt)\,dt&\text{ if $n$ is even}\\
        \int_{0}^{\delta}(\hat b+\frac12+a^nt)+\chi_b(\hat
        b+\frac12+a^nt)\,dt&\text{ if $n$ is odd}
      \end{cases}\\
      &=
      \begin{cases}
        \delta\hat b+\frac12a^n\delta^2+\delta-\delta\frac b{a-1}&\hspace*{1.75cm}\text{ if $n$ is even}\\
        \delta(\hat b+\frac12)+\frac12a^n\delta^2-\delta\frac b{a-1}&\hspace*{1.75cm}\text{ if $n$ is
          odd}
      \end{cases}
    \end{split}
  \end{displaymath}
  as long as $a^n|\delta|<\frac b{a-1}$, i.e. $n<\tilde
  N_{\delta}:=N_\delta-\frac{\log(a-1)-\log b}{\log a}$. For the remaining $n$
  this identity needs to be modified by at most $|\delta|$. In any case,
  \begin{displaymath}
    \sum_{n=0}^{N_\delta-1}\int_I(\chi_{b'}-\chi_{b})(x)\cdot g_b(T_{a,0}^nx)\,dx
    =
    \frac{N_\delta}2\cdot\frac\delta2+\frac{N_\delta}2\cdot0+\O(\delta)
    =
    \delta\log|\delta|^{-1}\frac{1}{4\log a}+\O(\delta)
    \quad\text{for odd $a$}
  \end{displaymath}
  with a constant in ``$\O$'' that depends on $b$ and $a$ but not on $b'$.  If
  $a$ is even, then $T_{a,0}\hat b=\hat b$, and following the argument for odd
  $a$ and even $n$ we obtain
  \begin{displaymath}
    \sum_{n=0}^{N_\delta-1}\int_I(\chi_{b'}-\chi_{b})(x)\cdot g_b(T_{a,0}^nx)\,dx
    =
    N_\delta\cdot\frac\delta2+\O(\delta)
    =
    \delta\log|\delta|^{-1}\frac{1}{2\log a}+\O(\delta)
    \quad\text{for even $a$.}
  \end{displaymath}

  We turn to more general parameters $0<b<\frac12$.
  We have to estimate
  \begin{equation}
    \begin{split}
      s(\delta)
      :=&
      \sum_{n=0}^{N_\delta-1}\int_I(\chi_{b'}-\chi_{b})(x)\cdot g_b(T_{a,0}^nx)\,dx      =
      \sum_{n=0}^{N_\delta-1}\int_{\hat b}^{\hat
        b+\delta}g_b(T_{a,0}^nx)\,dx\\
      =&
      \sum_{n=0}^{N_\delta-1}\int_{0}^{\delta}g_b(T_{a,0}^n(\hat b)+a^nt)\,dt
    \end{split}
  \end{equation}
  The details of this estimate depend
  strongly on the distributional properties of the orbit of $\hat b$ under
  $T_{a,0}$ and we discuss only two particular but important cases where the
  situation does not become too complicated. First we look at such $b$ for
  which the orbit of $\hat b$ is eventually periodic but where the periodic
  part does neither contain $-\frac12$ nor $\hat b$. (This is a countable
  dense set of parameters.) In this case one can argue as above for
  $b=\frac12$ and odd $a$ and show that
  \begin{displaymath}
    s(\delta)
    =
    \delta\log|\delta|^{-1}\frac{C_b}{\log a}+\O(\delta)
  \end{displaymath}
  with $C_b=\int_I g_b(x)\,d\mu_b(x)$ where $\mu_b$ is the equidistribution on the
  periodic part of the orbit of $\hat b$. Exceptionally this may be zero, but
  typically it won't.

  Next we look at Lebesgue typical points $\hat b$, i.e. at Lebesgue typical
  parameters $b$. For fixed $\delta$ we interpret $s(\delta)$ as a random
  variable where randomness is introduced via the parameter $b\in
  (0,\frac12)$. We are going to show that the random variables $\delta^{-1}
  N_\delta^{-\frac12}s(\delta)$ converge in distribution to a mixture of
  Gaussians, \ie
  \begin{equation}
    \label{eq:CLT}
    \L(\delta^{-1}
    N_\delta^{-\frac12}s(\delta))\Rightarrow\int_0^1\N(0,\sigma^2_z)\,dz
    \quad\text{as }\delta\to0
  \end{equation}
  with suitable variances $\sigma^2_z>0$ that depend on the fixed parameter
  $a$. This shows that approximately $s(\delta)=
  \delta(\log|\delta|^{-1})^{\frac12}\,Z$ with a random variable $Z$ that is a
  mixture of Gaussians which depends only on the fixed integer parameter $a$.

As a first step we compare $s(\delta)$ to
  $\sum_{n=0}^{N_\delta-1}\delta\cdot g_b(T_{a,0}^n\hat b)$. As
  $g_b(x)=x+\chi_b(x)$, we can estimate the difference for the $x$- and the
  $\chi_b(x)$-constributions separately. For the $x$-contribution the
  difference is easily seen to be of order $\O(\delta)$. For the
  $\chi_b(x)$-contribution we estimate each of the last $L_\delta:=\lceil\frac3{\log a}\log N_\delta\rceil$
  terms of the sum by $2$ thus getting a contribution of order
  $\O(\delta\log N_\delta)$. For the remaining terms we note the following two
  estimates which are obvious from a short look at the graph
  of $T_{a,0}^n$:
  \begin{displaymath}
    \begin{split}
      m\{b\in I:|T_{a,0}^n(\hat b)-(-\frac12)|<a^n\delta\},\;
      m\{b\in I:|T_{a,0}^n(\hat b)-\hat b|<a^n\delta\}
      &\leq
      4a^n\delta\ ,
    \end{split}
  \end{displaymath}
  so that
  \begin{displaymath}
    \begin{split}
    &m\left\{b\in I:\,\exists n\in\{0,\dots,N_\delta-L_\delta-1\}\text{ s.th. }|T_{a,0}^n(\hat b)-(-\frac12)|<a^n\delta
    \text{ or }|T_{a,0}^n(\hat b)-\hat b|<a^n\delta\right\}\\
    &\leq
    \frac8{a-1}a^{-L_\delta}
    =
    \frac8{a-1}N_\delta^{-3}\ .
    \end{split}
  \end{displaymath}
It follows that $\sum_{n=0}^{N_\delta-L_\delta-1}\int_0^\delta \chi_b(T_{a,0}^n(\hat
b)+a^nt)\,dt= \sum_{n=0}^{N_\delta-L_\delta-1}\delta\chi_b(T_{a,0}^n\hat
b)$ except on a set of $b$ of Lebesgue measure at most $\frac8{a-1}N_\delta^{-3}$.
Hence, observing that $L_\delta N_\delta^{-1}\to0$ as $\delta\to0$, the
convergence in \eqref{eq:CLT} will follow once we have proved
\begin{equation}
  \label{eq:CLTbis}
  \L(Y_{N_\delta})\Rightarrow\int_0^1\N(0,\sigma^2_z)\,dz\quad\text{as
  }N_\delta\to\infty
\end{equation}
where $Y_N(b):=N^{-\frac12}\sum_{n=0}^{N-1}g_b(T^n\hat b)$ and $b$ is
uniformly distributed in the interval $(0,\frac12)$. (To ease the notation we
abbreviate $T_{a,0}$ by $T$.) As the single contributions to the sum in $Y_N$
depend on $b$ via $T^n(\hat b)$ and $b$ itself, this is not the situation of
the usual central limit theorem, so we treat the problem in two steps:\\[1mm]
\emph{Step 1:} For fixed $z\in(0,1)$ consider
$Y_N^z(b):=N^{-\frac12}\sum_{n=0}^{N-1}g_z(T^n\hat b)$.  It is a well known
general fact that, for fixed $z$, the $Y_N^z$ converge in distribution to some
$\N(0,\sigma^2_z)$ (see \eg \cite{keller-80,HK-82,RE-83}) - except for the
strict positivity of $\sigma^2_z$. To prove this we use \cite[Lemma 6]{RE-83}:
suppose for a contradiction that $\sigma^2_z=0$. Then there is a function
$\psi:I\to\rz$ of bounded variation such that $g_z(x)=\psi(Tx)-\psi(x)$ for
Lebesgue-a.e. $x$. Let $M:=2\sup|\psi|$.  Then
$\left|\sum_{k=0}^{n}g_z(T^kx)\right|\leq M$ for all $n$ and a.a.  $x$.
Looking at suitable periodic orbits it is easy to see that there are $n=n_0$
and $x=x_0$ for which this sum is larger than $M+2$. But then, as both $T$ and
also $g_z$ are at least one-sided continuous, there is a small interval close
to $x_0$ on which the same sum is larger than $M+1$, which contradicts the
above bound that holds for all $n$ and Lebesgue-a.a. $x$.
\\[1mm]
\emph{Step 2:} We need a number of preparations:
\begin{enumerate}[(i)]
\item\label{list:i} Let $J:=(-\frac12,-\frac12+\frac1{2(a-1)})$ be the interval through
  which $\hat b$ ranges when $b$ is chosen randomly from $(0,\frac12)$.
\item\label{list:ii} Let $(r_j)_j>0$ be any sequence of natural numbers tending to
infinity and such that $r_j\leq j^{\frac14}$ for all $j$. For each $j$ denote
by $C_j\subset I=[-\frac12,\frac12]$ a set of points that subdivides $I$ into
$a^{r_j}$ intervals of the same length which are mapped onto $I$ bijectively
by $T^{r_j}$. (If $a$ is odd take $C_j:=T^{-r_j}\{-\frac12\}$, if $a$ is even
take $C_j:=T^{-r_j}\{0\}$.) For $z\in C_j$ denote by $I_z^j$ the subinterval
with left endpoint $z$.
\item\label{list:iii}
$\|Y_N-Y_N\circ T^{2r_N}\|_\infty=\O(r_NN^{-\frac12})=\o(1)$ as $N\to\infty$
because the two sums involved differ only by $4r_N$ terms.
\item\label{list:iv}
  $\Var(P_T^{r_N}(\frac1{m(J)}\1{J}-1))\leq(\frac2a)^{r_N}\Var(\frac1{m(J)}\1{J}-1)=\o(1)$
  as $N\to\infty$ by Lemma~\ref{lemma:var-contraction}.
\item\label{list:v} $m\{x\in I_z^N:Y_N^z(T^{r_N}x)\neq Y_N(T^{r_N}x)\}=a^{-r_N}m\{x\in
  I:Y_N^z(x)\neq Y_N(x)\}\leq a^{-r_N}m(I_z^{N})$.
\end{enumerate}
In order to prove \eqref{eq:CLT} we now proceed as follows: it suffices to show that for each bounded
Lipschitz function $\phi:\rz\to\rz$ holds
\begin{equation}
  \label{eq:Lip-convergence}
  \frac1{m(J)}\int_J\phi(Y_N(b))\,db
  =
  \int_0^1\left(\int_I\phi(b)\,d\N(0,\sigma^2_z)(b)\right)dz+\o(1)\quad\text{as
  }N\to\infty\ .
\end{equation}
To simplify the notation we write $\int_J\phi(Y_N)\,dm$ instead of
$\int_J\phi(Y_N(b))\,db$ \etc.
\begin{displaymath}
  \begin{split}
    \frac1{m(J)}\int_J\phi(Y_N)\,dm
    &=
    \frac1{m(J)}\int_J\phi(Y_N\circ T^{2r_N})\,dm+\o(1)\text{\qquad by
      (\ref{list:iii})}\\
    &=
    \int_IP_T^{r_N}(m(J)^{-1}\1{J})\cdot\phi(Y_N\circ T^{r_N})\,dm+\o(1)\\
    &=
    \int_I\phi(Y_N\circ T^{r_N})\,dm+\o(1)\text{\qquad by
      (\ref{list:iv})}\\
    &=
    \sum_{z\in C_N}\int_{I_z^N}\phi(Y_N^z\circ T^{r_N})\,dm+\o(1)\text{\qquad by
      (\ref{list:v})}\\
    &=
    \sum_{z\in C_N}\int_IP_T^{r_N}\1{I_z^N}\cdot\phi(Y_N^z)\,dm+\o(1)\\
    &=
    \sum_{z\in C_N}m(I_z^N)\int_I\phi(Y_N^z)\,dm+\o(1)\\
    &=
    \sum_{z\in C_N}m(I_z^N)\int_I\phi\,d\N(0,\sigma^2_z)+\o(1)\qquad\text{by step 1}\\
    &=
    \int_0^1\left(\int_I\phi\,d\N(0,\sigma^2_z)\right)dz+\o(1)\qquad\text{as }N\to\infty
  \end{split}
\end{displaymath}
where one has to choose a sufficiently slowly growing sequence $(r_N)$ in the
second last equality.


  \paragraph{Summary of results for integer $a$} In view of the factor
  $(a-1)^2$ in \eqref{eq:a-1square} and of the definition of
  $\delta=\frac{b'-b}{a-1}$ the above discussion shows:
  \begin{enumerate}[(1)]
  \item For even $a\ge4$,
    \begin{displaymath}
      \begin{split}
        D_a(b')-D_a(0)&=\O(b')\quad\text{ and }\\
        D_a(b')-D_a(\frac12) &=\frac{a-1}{2\log
          a}(b'-\frac12)\log|b'-\frac12|^{-1}+\O(b'-\frac12)
      \end{split}
    \end{displaymath}
    \label{obs:even}
  \item For odd $a\ge3$,
    \begin{displaymath}
      \begin{split}
        D_a(b')-D_a(0)&=-\frac{a-1}{2\log a}b'\log|b'|^{-1}+\O(b')\quad\text{ and}\\
        D_a(b')-D_a(\frac12)&=
        \frac{a-1}{4\log
          a}(b'-\frac12)\log|b'-\frac12|^{-1}+\O(b'-\frac12)
      \end{split}
    \end{displaymath}
    \label{obs:odd}
  \item If $b$ is such that $\hat b=-\frac12+\frac{b}{a-1}$ is eventually
    periodic under $T_{a,0}$ and the periodic part of the orbit neither
    contains $-\frac12$ nor $\hat b$, then there is a constant $C_{a,b}$ such
    that
    \begin{displaymath}
      D_a(b')-D_a(b)=C_{a,b}\,(b'-b)\log|b'-b|^{-1}+\O(b'-b).
    \end{displaymath}
    This
    generalizes observations~(\ref{obs:even}) and (\ref{obs:odd}). See also
    Remark~\ref{remark:Koza}.
    \label{obs:evper}
  \item For fixed $\delta$ and random $b$ drawn uniformly from $(0,\frac12)$ or from
    $(-\frac12,0)$,
    \begin{displaymath}
     D_a(b+\delta(a-1))-D_a(b)=C_a\,\delta(\log|\delta|^{-1})^{1/2}\,Z_\delta+\O(\delta) \text{ as }\delta\to0
    \end{displaymath}
    with a constant
    $C_a>0$ and random variables $Z_\delta$ which all have the same
    distribution - a mixture of
    Gaussians as in \eqref{eq:CLT} depending only on the fixed parameter $a$.
    \label{obs:random}
  \end{enumerate}
  In view of these findings the graphs of
  $D_a:[-\frac{1}{2},\frac{1}{2}]\to\rz$ are fractal in the sense of
  Section 12.2 in \cite{Tr-95} -- at least for integer values of $a$
  --, although they have box - and Hausdorff-dimension $1$.


\begin{remark}
In Refs.~\cite{Klages-96,GaKl-98,Klages-07} it has been shown
that for $b=0$ the dynamics of $\psi_\lambda$ can be expressed in
terms of generalized Takagi (or de Rham) functions. Analogous
conclusions hold for the case of $b\neq 0$ \cite{Klages-07b}.  The
above results are thus intimately related to continuity properties of
this class of functions under parameter variation.  These functions
are defined by simple functional recursion relations and have been
introduced in the literature completely independently from the
diffusion problem considered here.
\end{remark}

\section{Numerical results}
\label{sec:numerical}

Guided by the analytical results of the previous sections, in this
  part we numerically study the transport coefficients generated by
  the piecewise linear maps (\ref{eq:pwl}). We are aiming particularly
  at a numerical verification of Proposition~\ref{prop:main} and of
  the summary of Section~\ref{subsec:integer-slope}. To some
  extent the numerics enables us to go beyond the analytical
  findings as far as detailed local properties of these transport
  coefficients are concerned.

Let us start with a reminder of previous results: Exact analytical
solutions for drift $J$ and diffusion coefficient $D$ for all
parameter values $a,b$ were derived in
\cite{GrKl-02}.\footnote{Another set of formulas was reported in
\cite{Crist-06} but only for $D(a)$.} In \cite{KlKl-03} data sets were
generated from these expressions and analyzed by standard numerical
box counting \cite{Tr-95}. This procedure relies on the assumption
that

\begin{equation}
N(\epsilon)\sim\epsilon^{-B}
\label{eq:boxc}
\end{equation}
for small enough $\epsilon$, where $N$ is the number of square boxes
of side length $\epsilon$ needed to cover the graph of $J$ or $D$, and
$B=\dim_B(\graph)$ defines the box (counting) dimension.  Analysing
$D(a)=D(a,0)$ on $2\le a\le 8$, see Fig.~\ref{fig:fdimd1} (a), based
on $10^6$ data points uniformly distributed in $a$ yielded a box
dimension of $B\simeq1.039$ \cite{KlKl-03}. The inset in
Fig.~\ref{fig:fdimd1} (a) depicts $N(\epsilon)$ for the new, larger
data set of $10^9$ points of $D(a)$ in comparison to
(\ref{eq:boxc}) with the above exponent.  Data and fit are
undistinguishable.

Fig.~\ref{fig:fdimd1} (b) displays the numerical results from
\cite{KlKl-03} for the local box dimension $B(a)$ of $D(a)$. That is,
according to (\ref{eq:boxc}) $B$ was computed locally on a regular
grid of small subintervals $\Delta a$ centered around $a$. The figure
shows that (\ref{eq:boxc}) yields locally different results for
$B(a)$ forming an oscillatory structure that becomes more pronounced
the smaller $\Delta a$.  Consequently, $D(a)$ was said to be
characterized by a ``fractal fractal dimension'' in \cite{KlKl-03}.

\begin{figure}[t]
  \centerline{\includegraphics[width=10cm,angle=-90]{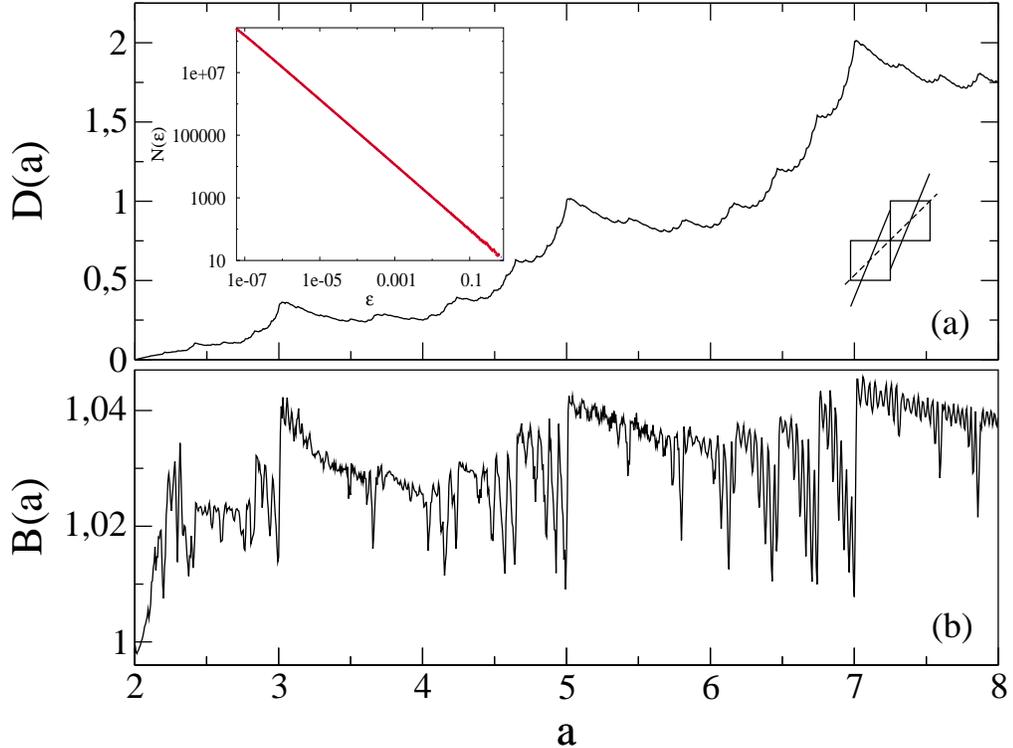}}
\caption{(a) Diffusion coefficient $D(a)$ on the interval $2\le a \le 8$ with
  $2000$ data points for the one-dimensional map (\ref{eq:pwl})
  sketched in the lower right edge with bias $b=0$ and $a$ as the
  slope of the map. The inset of (a) shows standard box counting based
  on $10^9$ data points for $D(a)$, where $N$ is the number of boxes
  of side length $\epsilon$. The dashed line in the inset depicts the
  power law (\ref{eq:boxc}) with box dimension $B=1.039$ as
  computed in \cite{KlKl-03}. (b) displays the box dimension $B(a)$
  computed locally on a regular grid of subintervals of size $\Delta
  a=0.006$ centered around $a$. For each subinterval a data set of
  $10^6$ values has been used, and a running average was performed
  over any three neighboring $B(a)$. Both figures, except the inset,
  are from \cite{KlKl-03}.}
\label{fig:fdimd1}
\end{figure}

\subsection{Box counting for the diffusion coefficient}

Motivated by Proposition~\ref{prop:main} and by \cite{Koza-04}, the
numerical results of \cite{KlKl-03} are now reevaluated and
supplemented by new, further numerical analysis. We start with the
diffusion coefficient $D(a)$.  Corollary~\ref{coro:main} states that
$B(a)=1$ for all intervals $\Delta a$, which is at variance with the
results presented in Fig.~\ref{fig:fdimd1}.  However, in contrast to
the standard box counting assumption (\ref{eq:boxc}),
Proposition~\ref{prop:main} is compatible with the existence of
multiplicative logarithmic terms by giving upper bounds for their
exponents. The discussion in Subsection~\ref{subsec:integer-slope}
shows that these terms do indeed exist.

In detail, Corollary~\ref{coro:main} states an upper bound for the box counting
function $N(\epsilon)$ of $D(a)$ of
\begin{equation}
N(\epsilon)\le K_4\epsilon^{-1}(1-\ln\epsilon)^2\quad .
\label{eq:Nbound}
\end{equation}
This motivates us to plot the product $N\epsilon$ as a function of
$-\ln\epsilon$: For small enough $\epsilon$ and in double-logarithmic
representation one should then see a straight line with the slope
yielding the exponent of the logarithmic term.  Fig.~\ref{fig:fdimd2}
numerically verifies the existence of this term for $D(a)$ on $2\le
a\le 8$: There clearly exists a non-zero exponent, however, in the
numerics $-\ln\epsilon$ is not large enough to overcome the additive
constant in (\ref{eq:Nbound}) for producing a straight line.

\begin{figure}[t]
  \centerline{\includegraphics[width=10cm,angle=-90]{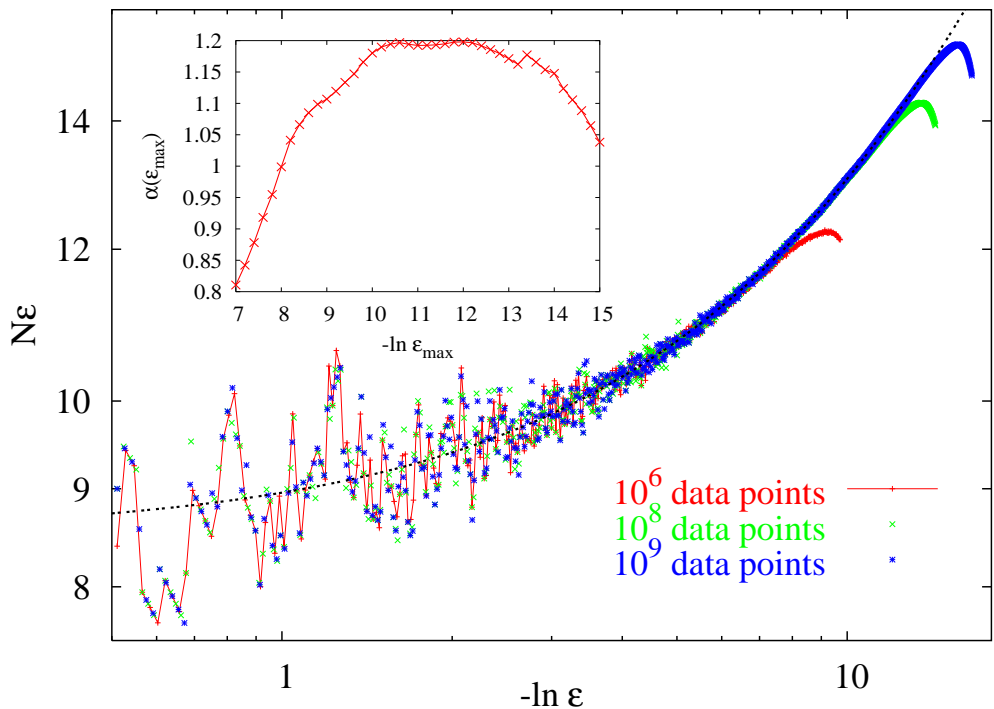}}
\caption{$N$ is the number of boxes of length $\epsilon$
  needed to cover $D(a)$ shown in Fig.~\ref{fig:fdimd1} (a) generated
  from $10^6$, $10^8$ and $10^9$ data points. Motivated by
  (\ref{eq:Nbound}), in contrast to the inset of
  Fig.~\ref{fig:fdimd1} (a) here we plot the product $N\epsilon$ as
  a function of $-\ln\epsilon$ double-logarithmically. The dashed
  black line represents a three-parameter fit for the largest data set
  over $0.5\le -\ln\epsilon\le12$ with the functional form of
  (\ref{eq:Nfit1}).  The inset shows results for the exponent of
  the logarithmic correction $\alpha$ obtained from fits where we vary
  the upper bound $\epsilon_{max}$ of the fit interval.}
\label{fig:fdimd2}
\end{figure}

In Fig.~\ref{fig:fdimd2} three data sets have been plotted consisting
of different numbers of data points for $D(a)$. The bending off of the
graphs at larger $-\ln\epsilon$ reflects that box counting starts to
resolve the single points of all the underlying data sets: From the
figure one can roughly estimate that for a data set of $10^6$ points
for $D(a)$ deviations set in around $-\ln\epsilon_{cut}\simeq7$, or
$\epsilon_{cut}\simeq10^{-3}$. Compared with a separation of $\delta
a=6\cdot 10^{-6}$ between any two data points along the $a$-axis, this
yields a difference of about three orders of magnitude. The same order
of magnitude argument holds if one compares $\epsilon_{cut}$ obtained
approximately for $10^8$ data points from Fig.~\ref{fig:fdimd2} with
the corresponding separation of $\delta a=6\cdot 10^{-8}$ between any
two data points. This leads to the prediction that for the set of
$10^9$ data points $-\ln\epsilon_{cut}\simeq14$ in
Fig.~\ref{fig:fdimd2}.

Inspired by (\ref{eq:Nbound}), we now fit the box counting results
with the function
\begin{equation}
N(\epsilon)= K_5\epsilon^{-1}(1+K_6\ln\epsilon)^{\alpha}
\label{eq:Nfit1}
\end{equation}
instead of (\ref{eq:boxc}). If this fit function reproduces the
numerically computed $N(\epsilon)$ reasonably well, Proposition~\ref{prop:main}
predicts that $0\le\alpha\le 2$. However, we emphasize that this
Proposition only gives us a strict upper bound -- it does not actually
tell us the ``true'' functional form of the whole graph. We have
indeed checked that fit functions others than (\ref{eq:Nfit1}),
which also obey (\ref{eq:Nbound}), work similarly well. In order
to be close to Proposition~\ref{prop:main} we stick to the fit function
(\ref{eq:Nfit1}) in the following.

The dashed black line in Fig.~\ref{fig:fdimd2} shows a fit of the box
counting results for $10^9$ data points of $D(a)$ with this functional
form.\footnote{For all fits the nonlinear least-squares
  Marquardt-Levenberg algorithm as implemented in gnuplot 4.0 has been
  used.} The inset of Fig.~\ref{fig:fdimd2} depicts results for the
  exponent $\alpha$ computed from different fit intervals
  $[0.5,-\ln\epsilon_{max}]$ for the same data set of $10^9$
  points. It indicates convergence towards $\alpha\simeq1.2 \;
  (-\ln\epsilon_{max}\to12)$. The decrease for $-\ln\epsilon_{max}>12$
  is well in agreement with the cutoff predicted above, which is due
  to the limited data set. Note that the cutoff sets in much later
  than the beginning of the plateau. Hence we conclude that for a data
  set of $10^9$ points for $D(a)$, $2\le a\le 8$, and by assuming the
  fit function (\ref{eq:Nfit1}), the numerical value for the
  exponent of the logarithmic term is $\alpha\simeq 1.2$. This is
  again in agreement with Proposition~\ref{prop:main}.\footnote{We have checked that
  these fit results do
  not significantly depend on the choice of the initial seeds for our
  three fit parameters and that the asymptotic standard error for them
  is less than 10\% for $-\ln\epsilon_{max}>10$. However, in our view
  quantitative error estimates are not reliable in this case, because
  we may not assume that the residua are normally distributed random
  variables.} Note that fits by (\ref{eq:Nfit1}) do not tell the
full story: The numerically exact data in Fig.~\ref{fig:fdimd2} show
the existence of a non-trivial fine structure pointing towards more
complicated functional forms for the ``true'' $N(\epsilon)$, which
should reflect the intricate structure of $D(a)$ in
Fig.~\ref{fig:fdimd1} (a). These irregularities may not be understood
as numerical errors.

\begin{figure}[t]

\vspace*{-5cm}
\subfigure{\centerline{\includegraphics[width=10cm,height=14cm,angle=-90]{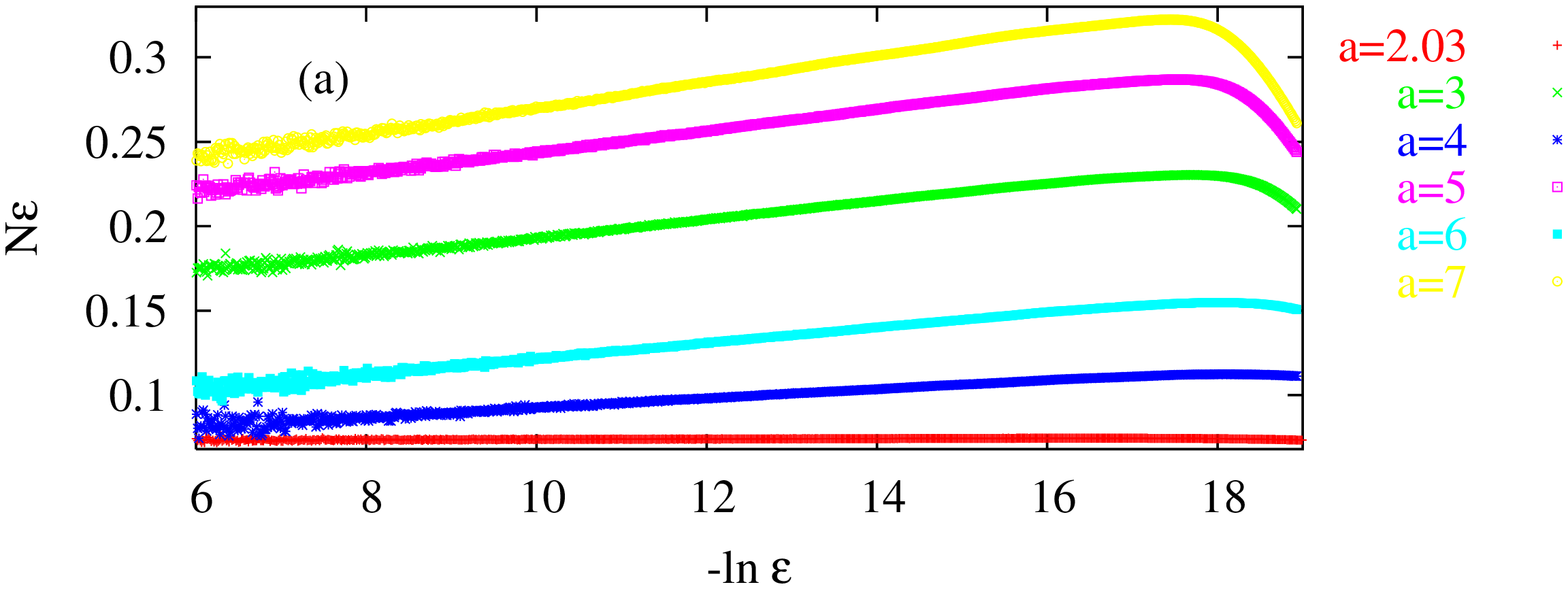}}}\\[-4ex]
\subfigure{\includegraphics[width=5cm,angle=-90]{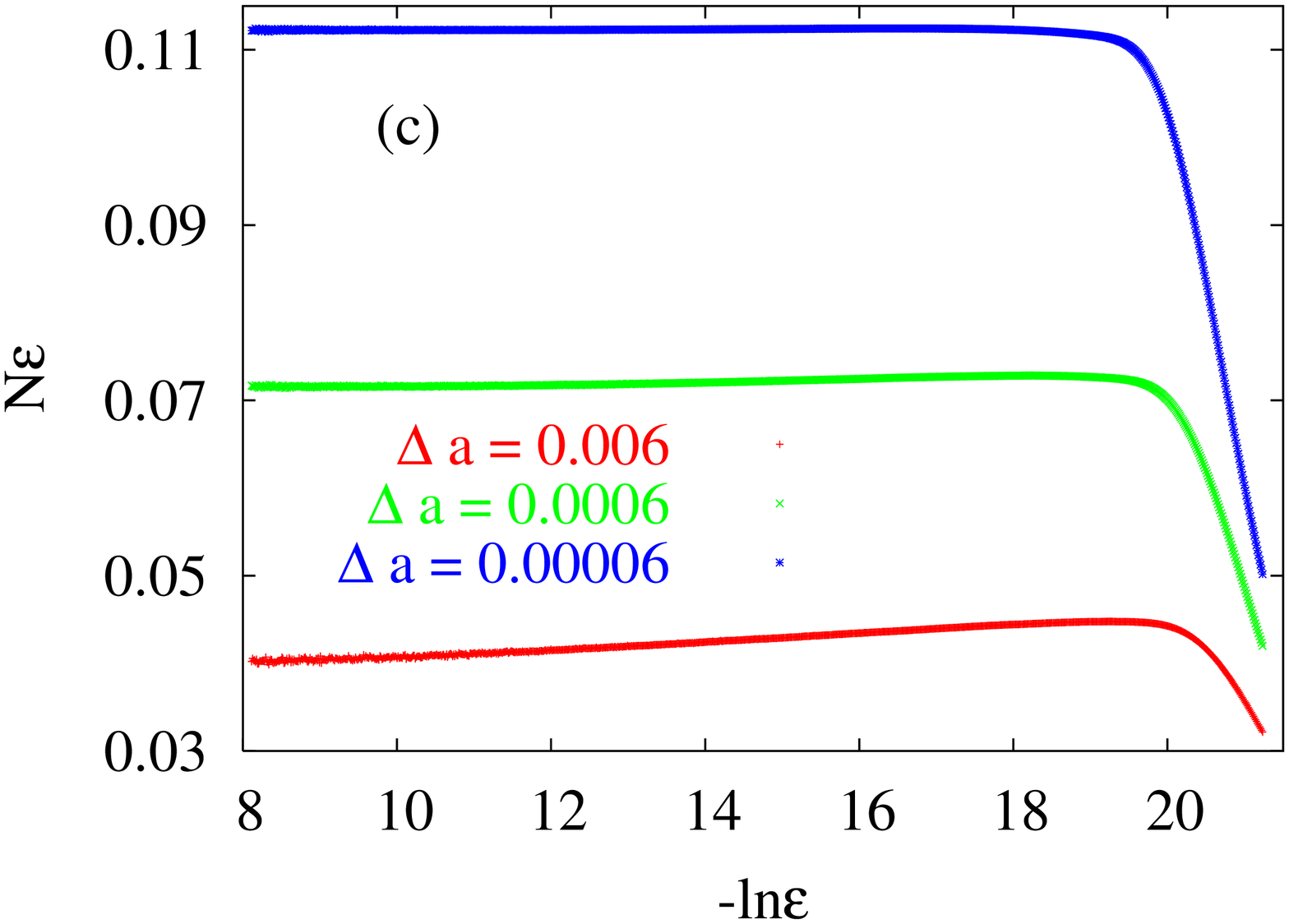}}
\subfigure{\includegraphics[width=5cm,angle=-90]{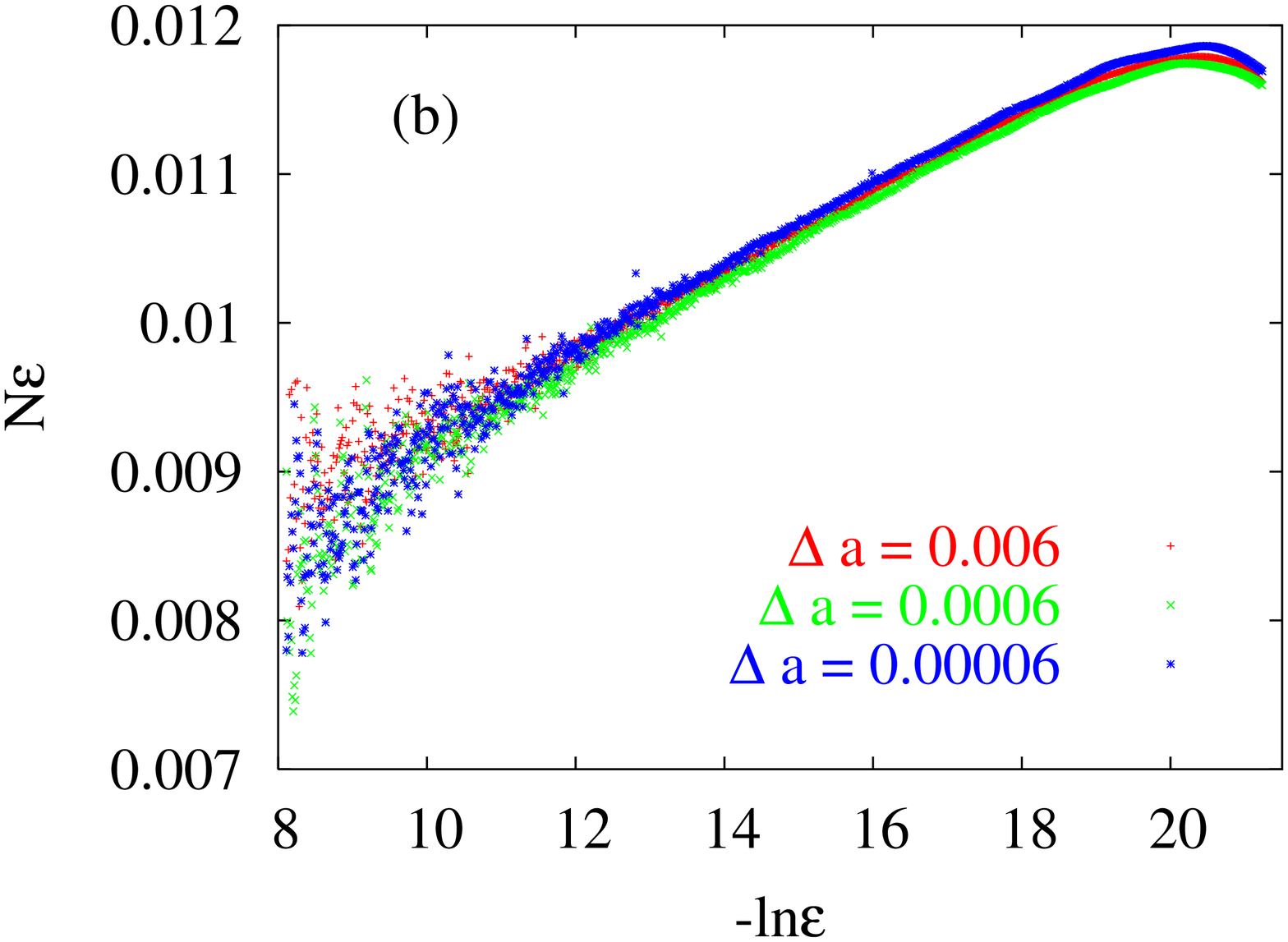}}
\caption{Local variation of the product $N\epsilon$ needed to cover $D(a)$
  around integer values of $a$: (a) shows results for parameter
  intervals of size $\Delta a=0.06$ centered around different $a$,
  based on $10^8$ data points.  (b) displays results for subintervals
  $\Delta a$ all centered around $a=4$, whereas in (c) all
  subintervals converge towards $a=5$. In (b) and (c) the graphs have
  been scaled by multiplying $\epsilon$ with the order of magnitude
  difference between the different values for $\Delta a$.}
\label{fig:fdimd3}
\end{figure}

After having verified the existence of logarithmic contributions on
large parameter intervals we now look at local variations of the
exponent $\alpha$. This is demonstrated by doing box counting for
$D(a)$ on small intervals around integer values of $a$.
Fig.~\ref{fig:fdimd3} (a) reveals that there exist two families of
curves: The one for even $a$ is at the bottom of this figure, whereas
the one for odd $a$ is on top. Additionally, all graphs show up
such that the ones for larger slopes are always on top in both
groups thus creating an oscillatory structure.

We first consider the special case $a=2.03$, where according to
Fig.~\ref{fig:fdimd3} (a) $\alpha\simeq0$. Note that $D(2)=0$,
correspondingly the parameter region just above $a=2$ marks the onset
of diffusion, cf.\ Fig.~\ref{fig:fdimd1} (a). As described in
\cite{Klages-96,KlDo-97,Klages-07}, for $a\to2$ there is asymptotic
convergence of $D(a)$ to the simple random walk solution
$D(a)=(a-2)/(2a)$. This physical argument explains why
$\alpha\to0\:(a\to2)$. There is a trend that larger even integer
slopes in (a) give $0\le\alpha\le 1$ whereas odd $a$ give
$1<\alpha\le2$.  Unfortunately, the fits producing these results are
very unstable, hence even these rough estimates should be taken with
care. In any case, the indicated order of magnitude of $\alpha$
appears to be in agreement with Proposition~\ref{prop:main}. Our fits furthermore
suggest that not only $\alpha$ is a function of $a$ but also that the
other two parameters in (\ref{eq:Nfit1}) are locally varying.
This agrees with conclusions drawn in \cite{KlKl-03}.

Figs.~\ref{fig:fdimd3} (b) and (c) provide a more detailed local
analysis by looking at successively smaller subintervals around two
specific slopes. While (c) suggests $\alpha\to 0\:(\Delta a\to0)$
around $a=5$, (b) with $a=4$ yields approximately $\alpha\to
1\:(\Delta a\to0).$\footnote{Again, the fit results are highly
  unstable, so the latter value should be taken with care.}  Note that
the graphs in (b) and (c) have been scaled as described in the figure.
Interestingly, this transformation leads to a collapse onto a master
curve in (b), whereas it does not work that way in (c).  Similar
observations have been reported in \cite{Koza-04}. Together with the
analytical results of Subsection~\ref{subsec:integer-slope},
Fig.~\ref{fig:fdimd3} thus demonstrates remarkable continuity
properties of $D(a,b)$ around integer slopes, which strongly depend on
the direction in parameter space.

These differences between graphs for odd and even $a$ are consistent
with the local box counting dimension displayed in
Fig.~\ref{fig:fdimd1} (b): They suggest that the oscillatory structure
in $B(a)$ actually reflects local variations of the parameters in
(\ref{eq:Nfit1}) determining the logarithmic corrections, erroneously
being fit in \cite{KlKl-03} with the standard box counting equation
(\ref{eq:boxc}) instead of taking the existence of logarithmic terms
into account. This result is confirmed by covering small parameter
regions around $a=4$ and $a=5$ with non-overlapping sequences of
subintervals and looking for local variations of the box counting
results. Again, one finds oscillations that roughly correspond to the
ones in Fig.~\ref{fig:fdimd1} (b). Although there is no linear
functional relationship between $\alpha(a)$ and $B(a)$, one may thus
argue that Fig.~\ref{fig:fdimd1} (b) tells us something about the
magnitude of local logarithmic corrections.

We remark that with the computing power available to us it was
impossible to produce a graph like Fig.~\ref{fig:fdimd1} (b) for local
values of $\alpha$, because for each $\alpha(a)$ box counting would
have required a data set of at least $10^9$ values of $D(a)$. Such
large data sets appear to be necessary because of monotonicity of the
exponents: if $E$ is a subset of $F$ then $\alpha(E)\le\alpha(F)$.
Local variations of $\alpha$ thus pose a serious problem to any
numerical box counting analysis, since eventually $\alpha(a)$ should
always converge to the largest local exponent.  However, if this
exponent is exhibited just on a tiny subinterval it could be extremely
tedious to detect it numerically. This argument of course also applies
to our previous result of $\alpha\simeq1.2$ for $D(a)$ on $2\le
a\le8$, which strictly speaking only holds for the given data set of
$10^9$ points. We cannot exclude that some tiny interval of $D(a)$
eventually yields a larger value of $\alpha$. In other words, the goal
of our numerical analysis cannot be to compute unambiguous values for
any exponents but rather to demonstrate qualitative and quantitative
order-of-magnitude agreement with Proposition~\ref{prop:main}.

\subsection{Box counting for the drift}

\begin{figure}[t]
\centerline{\includegraphics[width=13cm]{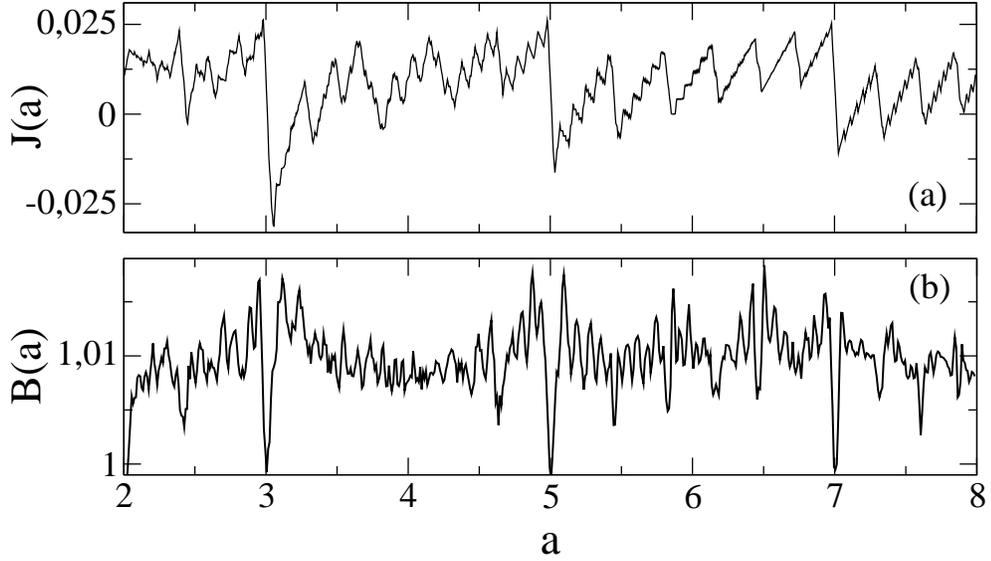}}
\caption{(a) Drift $J(a)=J(a,0.01)$ on the interval $2\le a\le8$
  based on 2000 data points. As in Fig.~\ref{fig:fdimd1}, (b) depicts
  the box dimension $B(a)$ computed locally on a regular grid of
  subintervals of size $\Delta a=0.01$ averaged over any three
  neighboring points. Both figures are from \cite{KlKl-03}.}
\label{fig:fdimd7}
\end{figure}

We continue our numerical analysis by investigating the parameter
dependence of the drift, or current, $J(a,b)$. As for the diffusion
coefficient, we start with a brief reminder of previous results in
form of Fig.~\ref{fig:fdimd7}: Like Fig.~\ref{fig:fdimd1}, it displays
a highly oscillatory structure both in the drift as well as in the
local box dimension as functions of $a$ for fixed $b$, where $B(a)$
has been computed according to (\ref{eq:boxc}) by again disregarding
any logarithmic corrections. Note particularly the pronounced minima
at odd integer values. As before, we now reevaluate these findings on
the basis of Proposition~\ref{prop:main} by taking logarithmic terms
into account.

\begin{figure}[t]
  \centerline{\includegraphics[width=10cm,angle=-90]{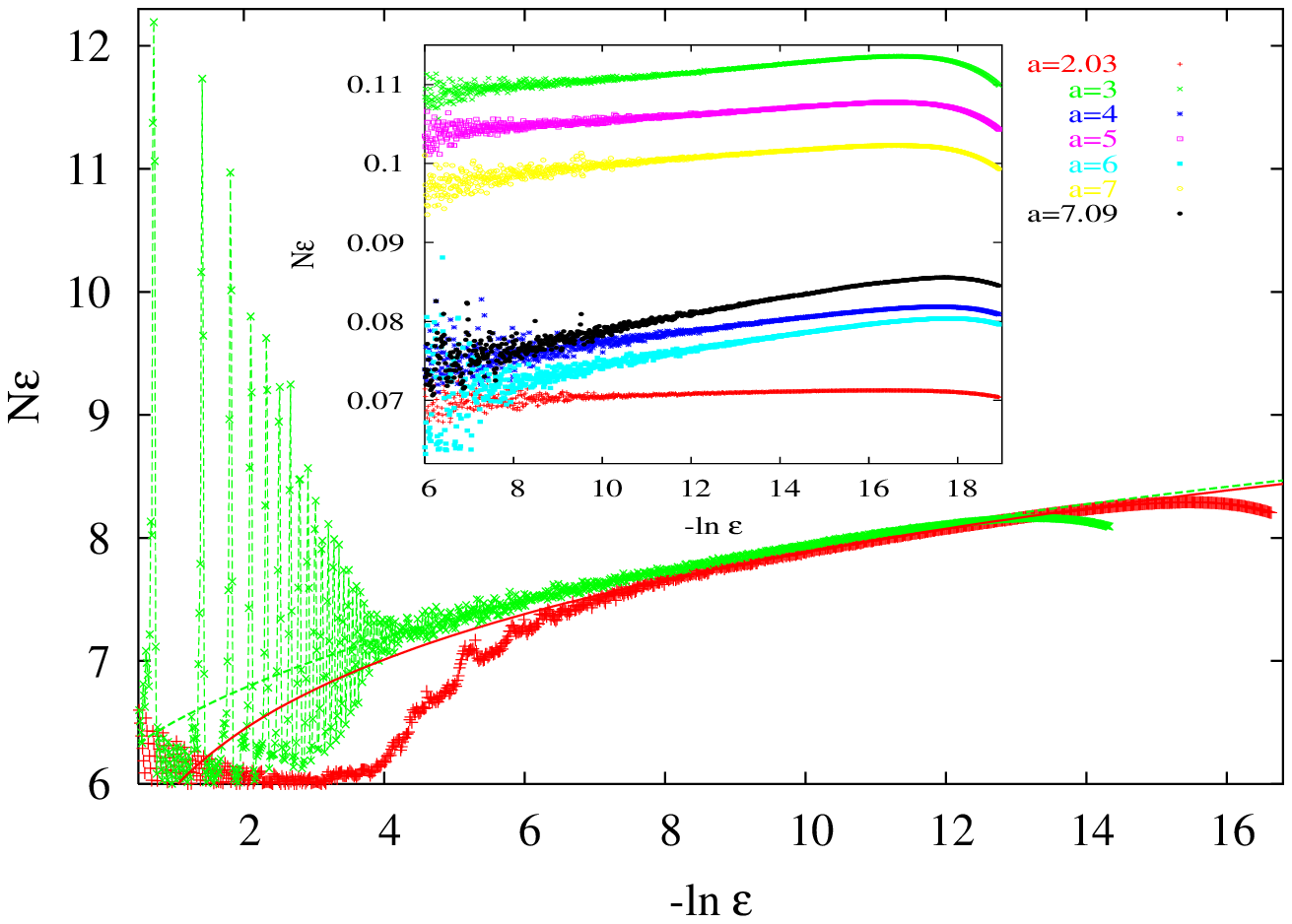}}
\caption{Main graph: Product $N\epsilon$ as a function of
  $-\ln\epsilon$ for the drift $J(a,b)$ over the interval $2\le a\le
  8$ at $b=0.01$ (red '+' symbols, based on $10^9$ data points) and at
  $b=0.49$ (green 'x' symbols, based on $10^8$ data points).  Included
  are two fits over the intervals $8\le-\ln\epsilon\le13$ ($b=0.01$)
  and $4\le-\ln\epsilon\le11.5$ ($b=0.49$). Inset: Local variation of
  the product $N\epsilon$ as a function of $-\ln\epsilon$ for
  parameter intervals of size $\Delta a=0.06$, mostly centered around
  integer values of $a$ and based on $10^8$ data points. The graph for
  $a=7.09$ demonstrates that, in agreement with Fig.~\ref{fig:fdimd7}
  (b), there exist strong local fluctuations of the box counting
  functions under variation of the slope $a$ of the map.}
\label{fig:fdimd8}
\end{figure}

Fig.~\ref{fig:fdimd8} numerically confirms the existence of
logarithmic corrections for $J(a,b)$: There exist non-zero exponents
$\alpha$ as allowed by Proposition~\ref{prop:main}.  Note particularly the
pronounced, different fine structures of both curves displayed in the
main part, which are much stronger than in Fig.~\ref{fig:fdimd2} for
$D(a)$. Due to these oscillations, in case of $J(a,b)$ it is
numerically very difficult to extract reliable values for the
exponents $\alpha$ by using (\ref{eq:Nfit1}). The two fits
included in the main graph yield an order of magnitude of
$\alpha\simeq0.1$, which matches to Proposition~\ref{prop:main}.

The inset of Fig.~\ref{fig:fdimd8} is analogous to
Fig.~\ref{fig:fdimd3} (a) in that it shows box counting results for
the current $J(a,b)\:,\:b=0.01$, mostly at integer values of the slope
$a$. Note that $J(2,0.01)\simeq0$ \cite{GrKl-02}, which marks the
onset of the drift. As we have argued for the diffusion coefficient,
at $a=2.03$ we are thus in a random walk regime for which one may
expect $\alpha\simeq0$, as is shown in the figure. However,
$\alpha(a)\neq0$ in all the other cases of the inset suggesting again
a local variability of $\alpha$ for $J(a,b)$, at least around integer
values of $a$. As in Fig.~\ref{fig:fdimd3} (a) there exist two family
of curves, one for even $a$ at the bottom and one for odd $a$ on
top of the figure. There is also again an additional ordering,
however, here it is such that curves for larger slopes are always at
the bottom in both families of graphs, except at $a=2.03$. The
additional graph for $a=7.09$ exemplifies the strong local variability
of $\alpha$ around $a=7$ which, as well as the difference between odd
and even slopes for box counting results of the drift, agrees with the
oscillations in the local box dimension $B(a)$ shown in
Fig.~\ref{fig:fdimd7}.

Fits for all the inset curves yield a trend towards small exponents
around even and somewhat larger values around odd slopes with an order
of magnitude of $0\le \alpha\le 1$, which appears to be consistent
with Proposition~\ref{prop:main}. However, we emphasize again that these results
give only a rough indication for the numerical reasons discussed
above. Exact results are only available for special cases: As we have
discussed in Subsection~\ref{subsec:integer-slope}, $J(a,b)=b$ for
constant $a\in\mathbb{N}$ under variation of $b$, where we thus have
$\alpha=0$, linear response and a caricature of Ohm's law. For general
$a$ one finds that $J(a,b)/(b|\log |b||)$ is bounded but has no limit
for $b\to0$ \cite{GrKl-02} pointing towards logarithmic corrections.

We have also qualitatively checked graphs of $D(a,b)$ and $J(a,b)$ for
other parameter values, that is, by choosing different values for $a$
and $b$ fixed in the parameter plane and studying the resulting
functions of the remaining free control parameters. Qualitatively, we
obtain results that are analogous to the ones discussed above.

\subsection{Continuity properties of the diffusion coefficient at integer slopes}
\label{subsec:numerics-integer-slope}

The previous two subsections demonstrated a very peculiar behaviour of
local box counting results for drift and diffusion coefficient around
integer slopes $a$ at fixed values of the bias $b$.
Subsection~\ref{subsec:integer-slope}, in turn, gave exact analytical
expressions for the difference $D_a(b')-D_a(b)$ of the diffusion
coefficient as a function of $\Delta b=b'-b$ at integer $a$ in the
limit of small $\Delta b$. This suggests to numerically study the
continuity properties of $D_a(b)$ at fixed integer values of $a$ in
more detail.

In order to access suitably small values of the parameter $\Delta b$,
we have employed the Fortran90 library mpfun90 \cite{Bailey-95} for
arbitrary-precision arithmetic. Using this library we have calculated
the difference quotient $(D_a(b')-D_a(b))/(b'-b)$ of $D$ with fixed
$b$ at values of $\Delta b$ down to
$10^{-200}$. Figure~\ref{fig:variation_in_b} (a) shows a subset of our
results for $a=3$ and $a=4$ at fixed $b \in \{-0.5,0\}$. There is
excellent agreement between the numerical results and the analytical
observations~(\ref{obs:even}) and~(\ref{obs:odd}) of
Subsection~\ref{subsec:integer-slope} predicting straight lines. This
agreement is as good to the limits of attainable precision, and has
been checked for other integer values than those shown in
Fig.~\ref{fig:variation_in_b}.

Figure~\ref{fig:variation_in_b} (b) depicts the diffusion coefficient
$D_a(b)$ at $a=4$ and a blowup around $b=0$, which corresponds to the
two curves in (a) at this $a$ value. Note that there is reflection
symmetry for $D_a(b)$ with respect to $b=-0.5$ and $b=0$. One can see
that at $b=-0.5$, where the difference quotient in (a) displays a
multiplicative logarithmic term, $D_a(b)$ in (b) exhibits a global
maximum in form of a sharp cusp. The global minimum at $b=0$, on the
other hand, is approached in a rather smooth, oscillatory manner
yielding a rounded-off shape, see the inset in (b). This relates to
the difference quotient curve in (a) with zero logarithmic
term. Analogous observations are made for $a=3$, where $D_a(b)$
exhibits local maxima both at $b=-0.5$ and at $b=0$, and for other
integer slopes. We remark that the quite regular structure of $D_a(b)$
in (b), particularly around both local extrema, resembles very much
the one of the fractal generalized Takagi functions studied in
\cite{Klages-96,GaKl-98,Klages-07}.

Observation~(\ref{obs:evper}) generalizes observations~(\ref{obs:even})
and~(\ref{obs:odd}) by stating that logarithmic corrections are typical
for parameter values of $b$ yielding Markov partitions. In
\cite{KlDo-95,Klages-96,KlDo-99,GrKl-02} it has been shown (for $b=0$)
that Markov partition parameter values identify local maxima and
minima of the parameter-dependent diffusion coefficient by relating
them to ballistic and localized orbits of the critical points of the
lifted map, respectively. One may thus speculate that the above
numerical observation holds true for local extrema on finer scales,
that is, that local cusps in $D_a(b)$ reflect logarithmic corrections
in the local difference quotient, whereas rounded-off local extrema
signify the lack of logarithmic terms. See also \cite{Koza-04} for
related results.  Furthermore, in Fig.~\ref{fig:variation_in_b} (a) we
deliberately restricted the range of $\ln (b'-b)$ so that, upon very
close scrutiny, a fine structure of all curves can be seen on top of
the straight line behaviour. Fig.~\ref{fig:variation_in_b} (b)
suggests that this oscillatory fine structure, which yields higher
order corrections to the analytical results of
Subsection~\ref{subsec:integer-slope}, is induced by the fine
structure of $D_a(b)$.

\begin{figure}[t]
\subfigure{\hspace*{-0.2cm}\includegraphics[width=5.5cm,angle=-90]{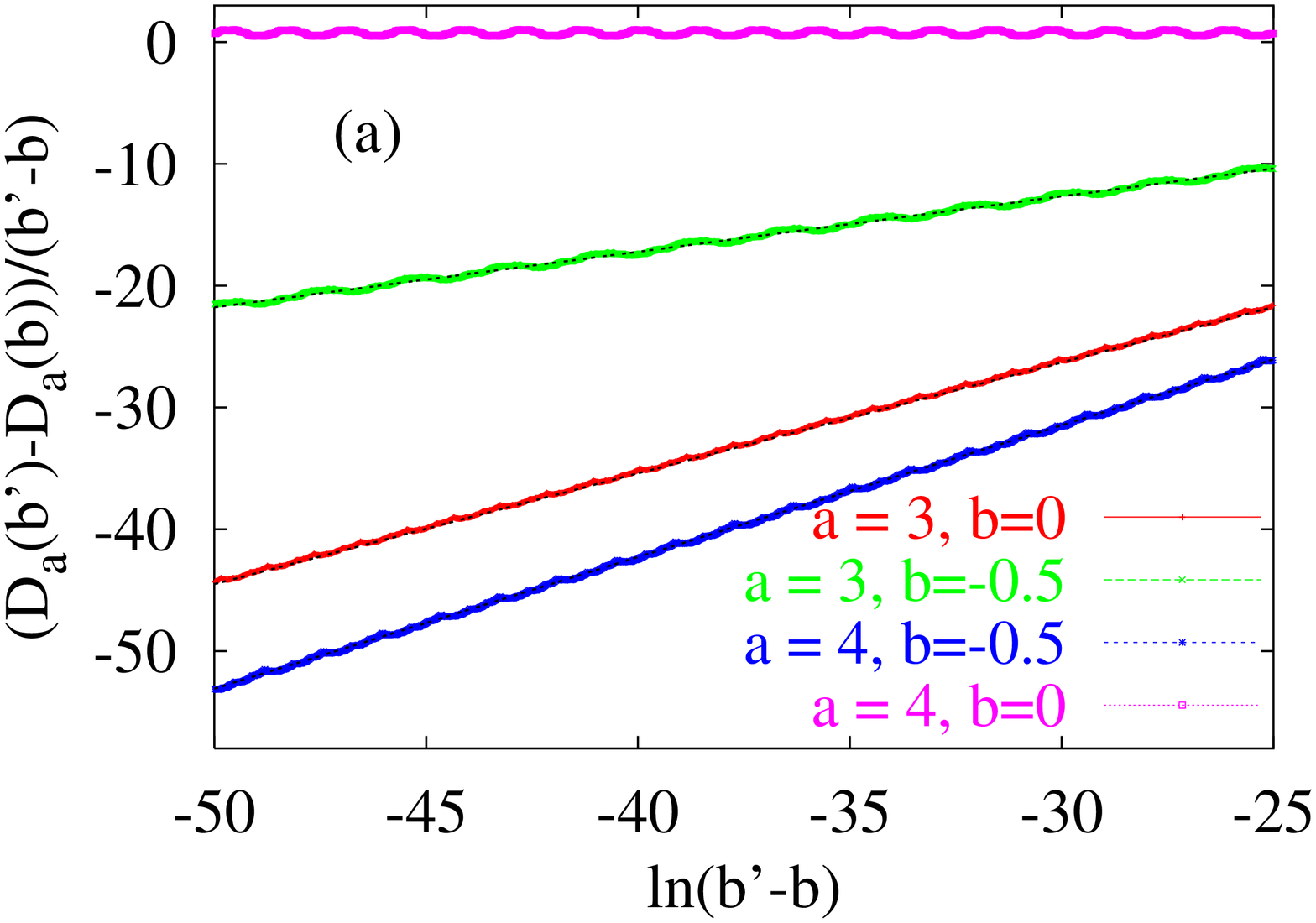}}
\subfigure{\hspace*{-0.5cm}\includegraphics[width=5.5cm,angle=-90]{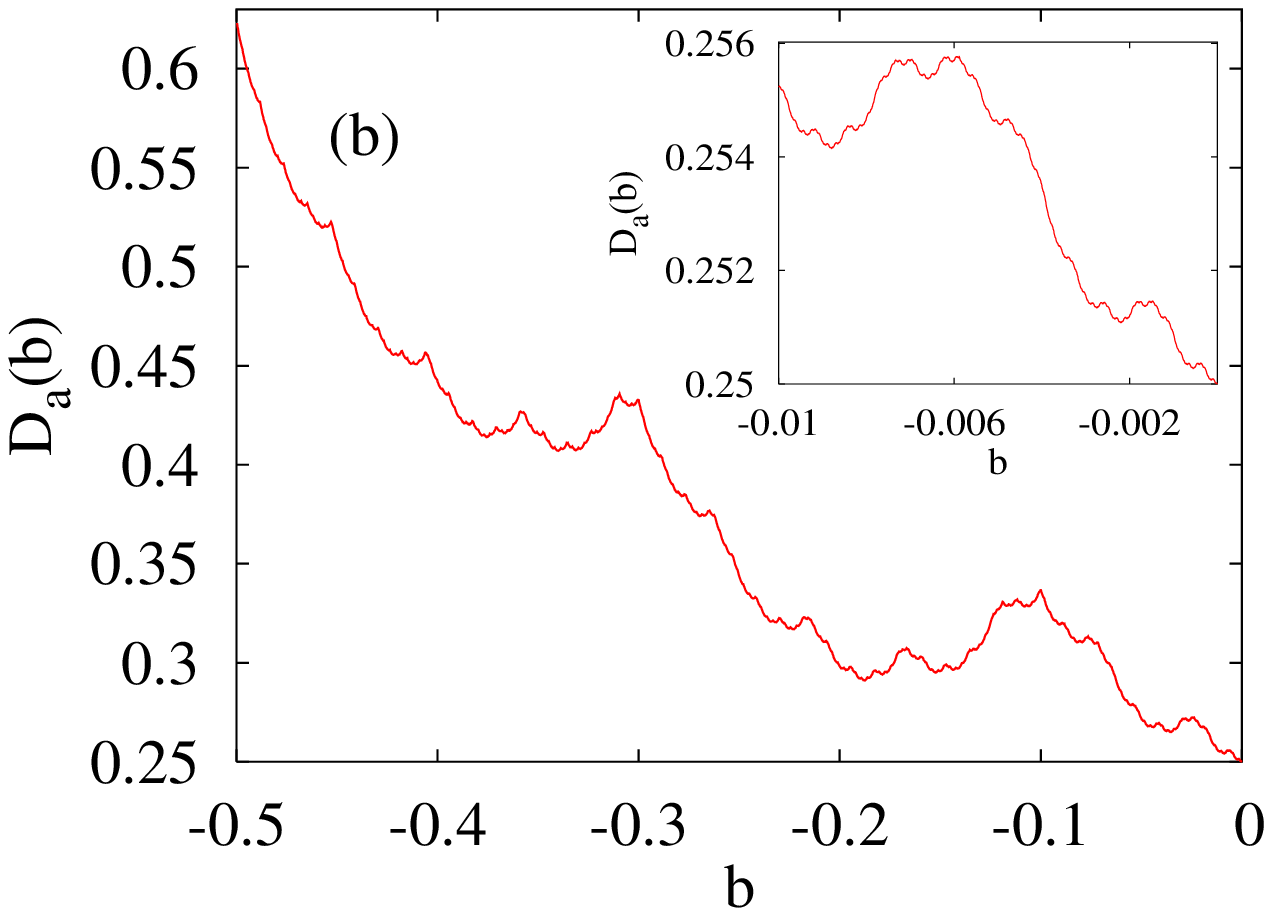}}
\caption{(a) Difference quotient $(D_a(b')-D(b))/(b'-b)$ as a function
of $\ln (b'-b)$ at integer values of $a \in \{3,4\}$ and with fixed $b
\in \{-0.5,0\}.$ Each curve is based on $10^5$ data points.  Included
are best fit curves (black dashed lines) whose fitted slopes (from
bottom to top: $3/(2\ln 4)$, $1/\ln 3$ and $1/(2\ln 3)$) agree with
the analytic predictions of Subsection~\ref{subsec:integer-slope} to
four significant figures.  The case $a=4,b=0$ clearly has slope zero,
as predicted. The barely visible fine scale oscillations of each curve
reflect higher order correlations in these quantities. (b) Diffusion
coefficient $D_a(b)$ at $a=4$ for $-0.5\le b\le0$ and a magnification
of the region around $b=0$. For each curve 2000 data points have been
computed from exact analytical solutions for $D_a(b)$
\cite{GrKl-02}. These curves form the basis for the two graphs at
$a=4$ displayed in (a).}  \label{fig:variation_in_b} \end{figure}

We have also numerically investigated the accuracy of
observation~(\ref{obs:random}) in
Subsection~\ref{subsec:integer-slope}.
Its main statement is that at fixed $\Delta b$ and with $b$ values taken uniformly
from the interval $[0,1/2),$ the quantity $(D(b')-D(b))/\Delta
b\sqrt{-\ln \Delta b}$ should be distributed like a mixture of
centered Gaussians, that distribution being independent of the particular value of
$\Delta b$. 
In fact what is typically seen at integer slopes is a distribution rather close to a pure Gaussian. We have tested this using the technique of quantile-quantile plotting (qqplots) as well the standard Shapiro-Wilk normality test. Both tools were implemented in the statistical package R \cite{R-07}.

Figure~\ref{fig:qqplot_a_4} presents results obtained for three sets
of data with the slope fixed at $a=4.$ For larger $a,$ the results
become closer to a fixed Gaussian, as the function $g(x)$ in
(\ref{eq:2-3-sum}) becomes more dominated by the $x$ term which has no
$b$ dependence. Here however, deviations from Gaussianity can be seen,
at least for sufficiently small $\Delta b.$ In the three parts of
Figure~\ref{fig:qqplot_a_4}, the red line with slope $\sigma$ and zero
offset $\mu$ shows the theoretical result for a Gaussian distribution
with standard deviation $\sigma$ and mean $\mu,$ with those parameters
here taken as those of our data set. As can be seen, all our
distributions show close agreement with this curve. However, the
Shapiro-Wilk normality test is more discerning: in (a) $\Delta
b=10^{-10}$ and we obtain a p-value of only $0.008,$ well below the
significance level for rejecting the null hypothesis of normality. In
(b) $\Delta b=10^{-50}$ and we get a p-value of $0.25,$ demonstrating
that this distribution is indeed very close to a pure Gaussian. It is
however likely that the deviations from Gaussianity in (a) are rather
due to deterministic effects arising from the relatively large value
of $\Delta b$ chosen and not from the nature of the true limiting
distribution being a mixture of Gaussians predicted by
observation~(\ref{obs:random}). It seems that the dominant behaviour
when the distribution seems to have converged is not detectably
different from a pure Gaussian.  We note that despite this, the two
distributions in (a) and (b) are similar and both have mean close to
zero, demonstrating that we have no disagreement with
observation~(\ref{obs:random}), merely that its details are too
sensitive to check numerically.

It is however possible to go further numerically, for example one can also study the nature of the distribution obtained when $b$ is taken
from a subinterval of $[0,1/2),$ which as can be seen in
Fig.~\ref{fig:qqplot_a_4} (c) leads in the case of integer $a$ to
distributions with rather more fine structure than the nice curves
seen in Figs.~\ref{fig:qqplot_a_4} (a) and (b). This is clear evidence of the deterministic nature of the underlying system in the form of strong correlations at fine scales. In this case the Shapiro-Wilk p-value is about $0.001$.

As far as observation~(\ref{obs:random}) is concerned, away from
integer values of $a$ quite different behaviour is seen thus
clarifying that this observation is rather to be considered
atypical.  Here the distribution of $D-$differences is centred
around zero still, but with a more sharply peaked and heavily tailed
distribution than a Gaussian. These deviations persist even very
close to the integer cases (e.g.\ at $a=3+10^{-50}$), though Gaussian behaviour does appear to be approached slowly in the limit
of integer values.

These numerical methods can also be used to investigate variation of
the continuity of the transport coefficients as $b$ is held fixed and
$a$ varies, as considered in \cite{Koza-04} and already looked at
using box counting in Fig.~\ref{fig:fdimd3}. Here the maximal exponent
of logarithmic correction, i.e.\ $D(a')-D(a) \sim
|a'-a|(\ln|a'-a|)^2,$ can be seen for odd $a,$ and though this might
appear to be in contradiction to the third part of
Fig.~\ref{fig:fdimd3} for $a=5,$ in fact arbitrarily close to $a=5$
the exponent tends locally to zero. Thus the box counting only sees
the ``typical'' local behaviour and the current method is more suited
for picking out atypical behaviour at specific points.


\begin{figure}[t]
\centerline{\subfigure[]{\includegraphics[width=4.5cm]{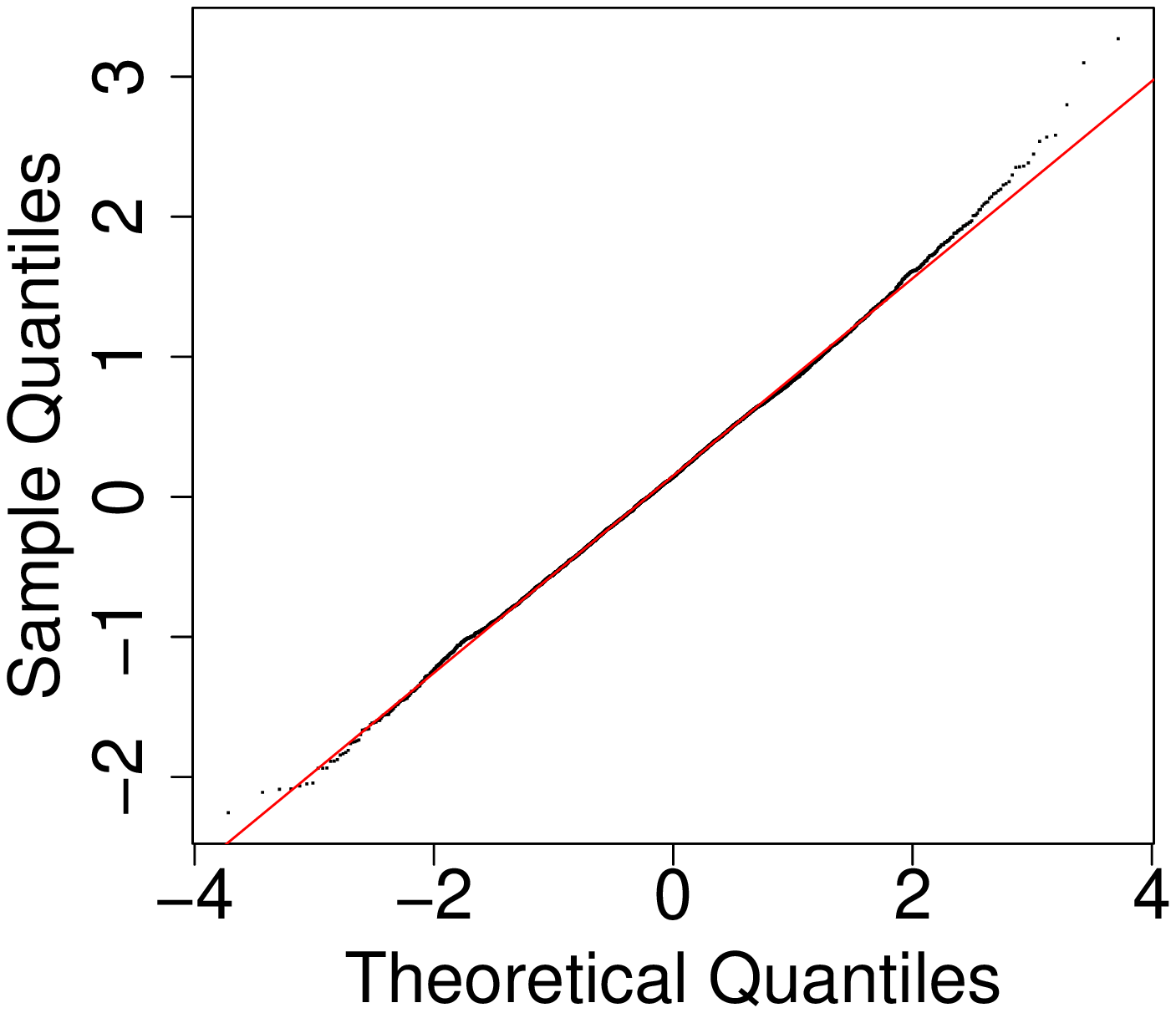}}
\subfigure[]{\includegraphics[width=4.5cm]{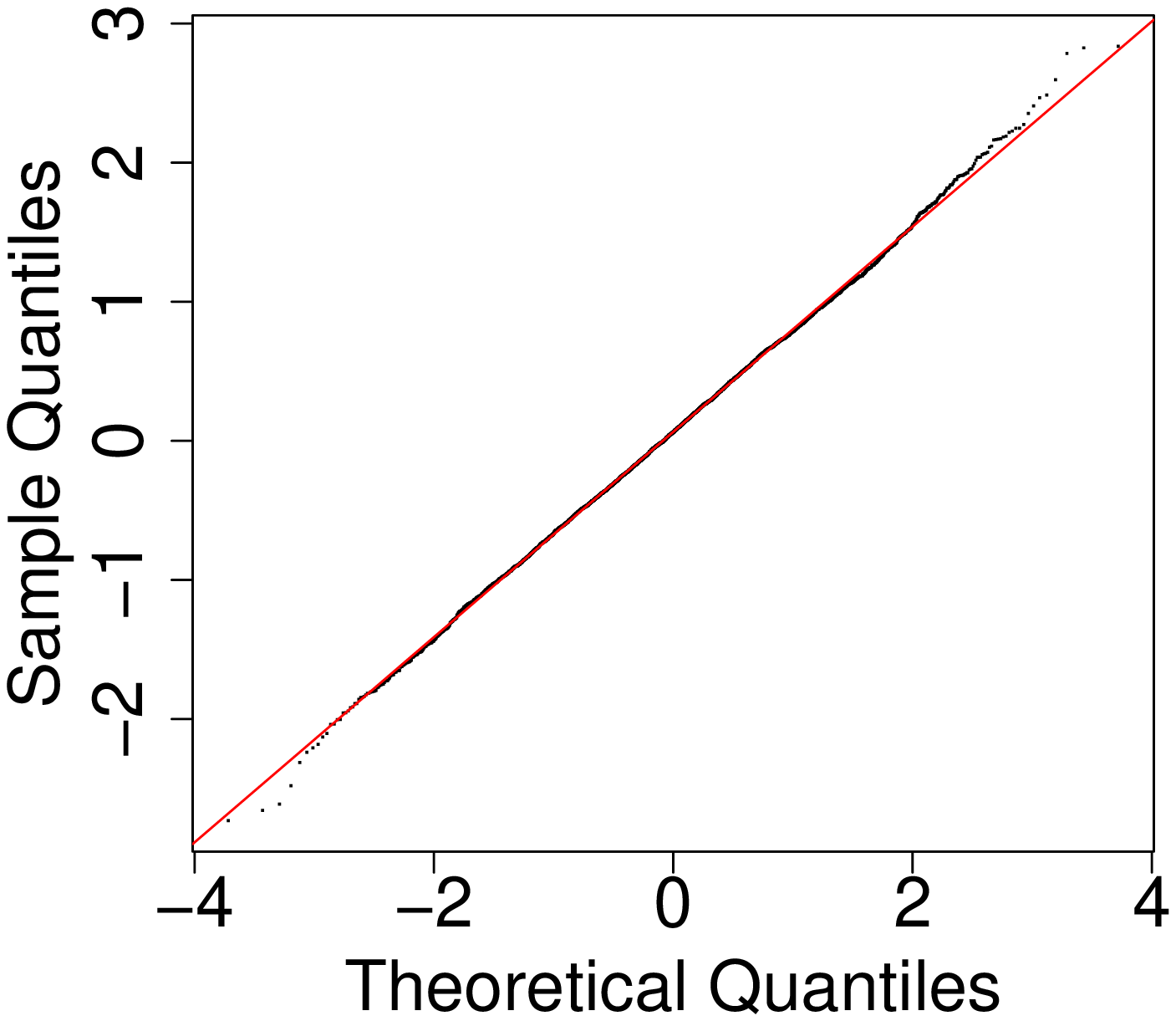}}
\subfigure[]{\includegraphics[width=4.5cm]{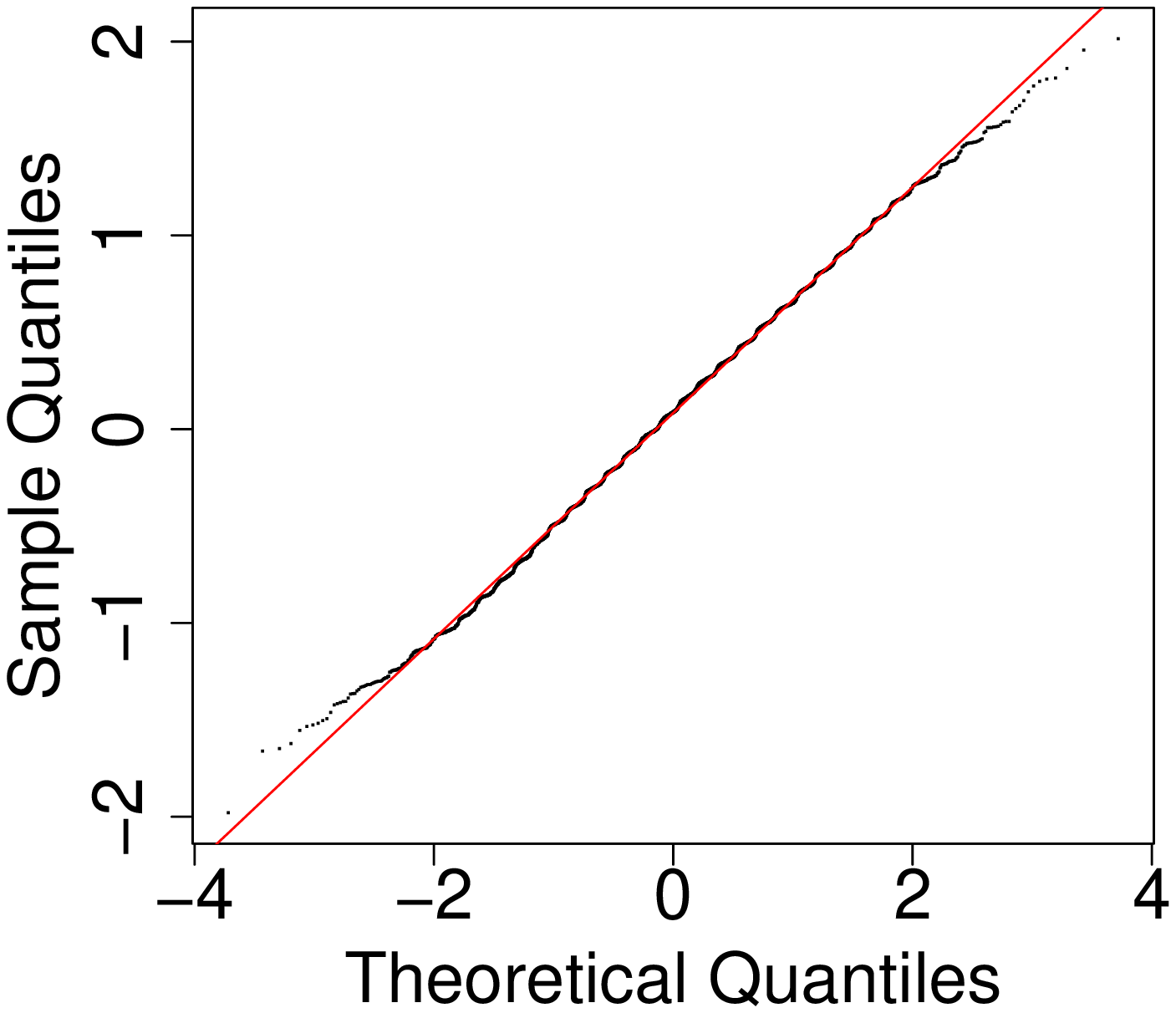}} }
\caption{(a) Normal quantile-quantile plot at $a=4$ for the
distribution of $D(b')-D(b),$ scaled by $\Delta b\sqrt{-\ln \Delta
b},$ with chosen $\Delta b=b'-b=10^{-10}$ held constant and $b$
picked from a uniform distribution on $[0,1/2).$ The red line with
slope $\sigma$ and zero offset $\mu$ would be the result for a
Gaussian distribution with standard deviation $\sigma$ and mean
$\mu.$ Here the numerically obtained values of these parameters were
used for the fit. (b) As (a) but with $\Delta b=10^{-50}$. (c) As
(b) but with the range of $b$ restricted to $[0,0.005).$ }
\label{fig:qqplot_a_4}
\end{figure}

\section{Conclusions and outlook}

\begin{enumerate}[(1)]
\item We proved rigorously that the diffusion coefficient of
  deterministic random walks generated by piecewise expanding interval
  maps depends continuously on the maps. More precisely, for
  ``natural'' parametrizations of the maps by some parameter
  $\lambda$, the diffusion coefficient as a function of the parameters
  has a modulus of continuity not worse than
  $|\delta\lambda|(\log|\delta\lambda|)^2$. Even if all maps in the family are
  topologically conjugate, the detailed analysis of
  section~\ref{subsec:integer-slope} shows that the modulus of continuity cannot be
  expected to be better than $|\delta\lambda\cdot\log|\delta\lambda||$.
This is in sharp
  contrast to the situation for the drift (or other averages of
  observables) that depend differentiably on parameters in this case
  \cite{BaSm-07}. One might thus conjecture:
If the maps are all topologically conjugate as in the case
of integer slopes, then J is Lipschitz and D has simple logarithmic
corrections. Otherwise J has simple logarithmic corrections and D has
quadratic ones.

\item We verified numerically the existence of logarithmic corrections
  in the box counting data for both the parameter dependent drift and
  diffusion coefficients. The computed values for the exponents of
  these logarithmic terms are compatible with the bounds predicted by
  our mathematical theory. However, we emphasize again the serious
  difficulties to obtain quantitatively reliable numerical results,
  which required to analyze huge data sets. These difficulties are due
  to strong local variations of these exponents and of the other
  control parameters governing the logarithmic corrections, as we find
  numerically.

  These new numerical results correct and amend the previous box
  counting analysis of Klages and Klau{\ss} \cite{KlKl-03} along the
  lines conjectured by Koza \cite{Koza-04}. Our model thus generates
  interesting examples of fractals for which the definition of the
  standard box counting dimension is misleading. We conclude that the
  (local) non-integer variations of the box dimension reported in
  \cite{KlKl-03} actually reflect non-trivial local variations of the
  parameters of the logarithmic corrections.

  We have furthermore numerically verified analytical predictions for
  the difference quotient of the diffusion coefficient as a function
  of the bias at integer slopes. These results suggest that the
  existence of logarithmic corrections is intimately related to the
  shape of the extrema in the diffusion coefficient curves.

\item In \cite{KoKl-07} a nonlinear generalization of our present
  model has been studied, which exhibits anomalous diffusion generated
  by marginal fixed points. Computer simulations led to conjecture
  that the anomalous diffusion coefficient of this map is
  discontinuous on a dense set of parameter values. It would be
  interesting to check this conjecture mathematically.

  These fractal transport coefficients also seem to provide a nice
  testing ground for methods of multifractal analysis \cite{FaWi-06}.

  Another important problem is to check whether such logarithmic
  corrections in transport coefficients might also be expected to
  occur in more `physical' systems, which are perhaps even accessible
  experimentally. This seems to be strongly related to the question
  whether a family of physical dynamical systems shares the same
  topological conjugacy class under parameter variation.

\end{enumerate}

\noindent {\bf Acknowledgements:}\\
G.K.\ and R.K.\ thank C.~Beck, C.~Dettmann and M.~Pollicott, the
organizers of the LMS Durham Symposium on {\em Dynamical Systems and
  Statistical Mechanics} in July 2006, where this work was started,
for their kind invitation.  R.K.\ and P.J.H.\ were supported by a
grant from the British EPSRC under EP/E00492X/1.

\end{document}